\documentclass[3p,12pt]{mystyle}

\usepackage[T1]{fontenc}
\usepackage[latin1]{inputenc}

\usepackage{amsmath,amssymb,theorem}

\usepackage{amsfonts}
\usepackage{wasysym}

\usepackage{epsfig}
\usepackage{epic}
\usepackage{eepic}
\usepackage{graphics}
\usepackage{graphicx}
\usepackage{color}
\usepackage{longtable}
\usepackage{subfigure}
\usepackage{enumitem}
\usepackage{multicol}
\usepackage[ruled,vlined]{algorithm2e}

\newcommand{\ds}{\displaystyle}
\newcommand{\ts}{\textstyle}

\newcommand{\Ord}{\mathcal{O}}

\newcommand{\Nn}{{\mathbb N}}

\newcommand{\Rr}{{\mathbb R}}
\newcommand{\Cc}{{\mathbb C}}
\newcommand{\Zz}{{\mathbb Z}}
\newcommand{\Dd}{\mathbb{D}}

\newcommand{\vph}{\varphi} 

\renewcommand{\Re}{\operatorname{Re}}
\renewcommand{\Im}{\operatorname{Im}}

\newcommand{\vect}[1]{\boldsymbol{#1}}

\DeclareMathOperator{\Iop}{I}
\newcommand{\I}{\boldsymbol{\Iop}}
\newcommand{\Imm}{\I^{(\vect{m})}}

\DeclareMathOperator{\Jop}{J}
\DeclareMathOperator{\Kop}{K}
\newcommand{\J}{\vect{\Jop}}
\newcommand{\Jmm}{\J^{(\vect{m})}}
\newcommand{\Kmm}{\vect{\Kop}^{(\vect{m})}}

\DeclareMathOperator{\LSop}{RD}
\newcommand{\LS}{\boldsymbol{\LSop}}
\newcommand{\LSm}{\LS^{(\vect{m})}}

\newcommand{\Gam}{\vect{\Gamma}^{(\vect{m})}}
\newcommand{\Gamsquare}{\vect{\Gamma}^{(\vect{m})}_{\square}}
\newcommand{\Gamsquarezero}{\vect{\Gamma}^{(\vect{m})}_{\square,0}}
\newcommand{\Gamsquareone}{\vect{\Gamma}^{(\vect{m})}_{\square,1}}

\newcommand{\Gamtriangle}{\vect{\Gamma}^{(\vect{m})}_{\triangle}}
\newcommand{\GamD}{\vect{\Gamma}^{(\vect{m})}_{\Omega}}
\newcommand{\Gamcomp}{\vect{\Upsilon}^{(\vect{m})}}
\newcommand{\Gamsupsquare}{\overline{\vect{\Gamma}}^{(\vect{m})}_{\square}}

\newcommand{\rhodonea}{\vect{\varrho}^{(\vect{m})}_{\alpha}}

\newcommand{\Pim}{\Pi^{(\vect{m})}}
\newcommand{\Pimsquare}{\Pi^{(\vect{m})}_{\square}}
\newcommand{\Pimtriangle}{\Pi^{(\vect{m})}_{\triangle}}
\newcommand{\Pimreal}{\Pi^{(\vect{m})}_{\mathcal{R}}}
\newcommand{\Pimrealsquare}{\Pi^{(\vect{m})}_{\square,\mathcal{R}}}
\newcommand{\Pimrealtriangle}{\Pi^{(\vect{m})}_{\triangle,\mathcal{R}}}
\newcommand{\Pims}{\Pi^{(\vect{m})}_{\mathrm{D}}}
\newcommand{\Pimssquare}{\Pi^{(\vect{m})}_{\square,\mathrm{D}}}

\newcommand{\Dx}[1]{\mathrm{d}#1}

\newcommand{\polbas}{\chi^{(\vect{m})}_{\vect{\gamma}}}
\newcommand{\polbasreal}{\chi^{(\vect{m})}_{\mathcal{R},\vect{\gamma}}}
\newcommand{\polbascont}{{X}_{\vect{\gamma}}}
\newcommand{\polbascontreal}{X_{\mathcal{R},\vect{\gamma}}}

\newcommand{\funpol}{\mathcal{L}}
\newcommand{\funpols}{\mathcal{L}_{\mathrm{D}}}

\usepackage[scr=boondoxo,scrscaled=1.05]{mathalfa}
\newcommand{\ff}{\mathscr{f}}

\newtheorem{theorem}{Theorem}

\newtheorem{proposition}[theorem]{Proposition}
\newtheorem{corollary}[theorem]{Corollary}
\newdefinition{definition}{Definition}
\newdefinition{remark}{Remark}
\newdefinition{example}{Example}
\newproof{proof}{Proof}

\newcommand{\proofofref}{}
\newproof{zproofof}{Proof of \proofofref}
\newenvironment{proofof}[1]
 {\renewcommand{\proofofref}{#1}\zproofof}
 {\endzproofof}

\begin{document}

\begin{frontmatter}
\title{Rhodonea curves as sampling trajectories for spectral interpolation on the unit disk}

\author{Wolfgang Erb}

\address{ University of Padova \\
          Department of Mathematics "Tullio Levi-Civita" \\
          Via Trieste 63, 35121 Padova, Italy} 
	
\ead{erb@math.unipd.it}

\date{\today}

\begin{abstract}
Rhodonea curves are classical planar curves in the unit disk with the characteristic shape of a rose. 
In this work, we use point samples along such rose curves as node sets for a novel spectral interpolation scheme on the disk. By deriving a discrete orthogonality structure on these rhodonea nodes, we will show that the spectral interpolation problem is unisolvent. 
The underlying interpolation space is generated by a parity-modified Chebyshev-Fourier basis on the disk. This allows us to compute the spectral interpolant in an efficient way. Properties as continuity, convergence and numerical condition of the scheme depend on the spectral structure of the interpolation space. For rectangular spectral index sets, we show that the interpolant 
is continuous at the center, the Lebesgue constant grows logarithmically and that the scheme converges fast if the function under consideration is smooth. 
Finally, we derive a Clenshaw-Curtis quadrature rule using function evaluations at the rhodonea nodes and conduct some numerical experiments to compare different parameters of the scheme.
\end{abstract}


\end{frontmatter}

\section{Introduction} \label{sec:introduction}

Rose curves are classical planar curves in a disk that have the shape of a patelled rose. Guido Grandi, studying these curves profoundly in the beginning of the 18th century \cite{Grandi1728}, 
used the corresponding greek name for them: rhodonea curves. These algebraic curves have a particular simple parametric and polar form. This makes them to interesting trajectories for data sampling in imaging. Examples of scanning systems using rhodonea curves are, for instance, Magnetic Particle Imaging \cite{Knopp2009PhysMedBio,Knopp2017,Szwargulski2015b} and laser scanners based on rotating Risley prisms
\cite{DumaSchitea2018,Lu2014}.
Further, rose curves are also very popular in calculus text books to teach parametrization and integration in polar coordinates. 

In this work, we study rhodonea curves as sampling trajectories for new and promising sets of interpolation nodes on the unit disk. If the samples are taken in a time-equidistant way along the curve, the nodes form a pair of interlacing polar grids. This structure of the so called rhodonea 
nodes together with an accordingly chosen basis system allows us to construct a simple and efficient spectral interpolation and quadrature scheme on the disk. 

As a suitable basis system for the spectral interpolation on the rhodonea nodes we use a parity-modified Chebyshev-Fourier basis. Among other well-known basis systems 
as the Logan-Shepp ridge polynomials or the Zernike polynomials, the Chebyshev-Fourier basis is a very popular choice for spectral methods on the unit disk 
\cite{Boyd2000,BoydYu2011,Fornberg1995,Fornberg1996,Shen2011,Trefethen2000,TownsendWilberWright2017}. One main advantage of the Chebyshev-Fourier basis is the possibility to compute the interpolating function very efficiently using fast Fourier methods. In relation to other systems, this basis system performs however not so well at the center of the unit disk. For a detailed comparison of the different spectral methods on the unit disk we refer to the profound discussion in \cite{BoydYu2011}.

\vspace{2mm}

{\noindent \bfseries 1.1. Main contributions.} 
\begin{itemize}[nosep,leftmargin=1em,labelwidth=*,align=left]
\item[-]\emph{Characterization of the rhodonea interpolation nodes.} We provide new descriptions of the intersection and boundary points of the rhodonea curves and show how they can be used as nodes for a 
spectral interpolation scheme on the disk.
\item[-]\emph{Unisolvence of interpolation scheme on rhodonea nodes.} We will prove the unisolvence of the spectral interpolation problem on the rhodonea nodes. The interpolation spaces are spanned by
a general spectral set of Chebyshev-Fourier basis functions.
\item[-]\emph{Efficient implementation.} We show that the spectral interpolation on the rhodonea nodes can be performed efficiently using a two-dimensional fast Fourier transform. 
\item[-]\emph{Numerical condition and convergence analysis.} The main interpolation space considered in this work is based on a rectangular spectral index set. For this space we show that the numerical condition of the interpolation is growing only logarithmically in the number of nodes and that the scheme converges fast if the interpolated function is smooth. 
\item[-]\emph{Continuity and quadrature.} For the rectangular spectral index set we can guarantee that the interpolant is continuous at the center of the disk. Further, we show how the interpolation scheme can be used to define a Clenshaw-Curtis quadrature rule on the disk. 
\end{itemize} 

\vspace{2mm}

{\noindent \bfseries 1.2. Comparison to existing work.} 

\noindent \emph{Comparison to standard tensor-product schemes on the disk.} 
Spectral methods based on a Chebyshev-Fourier basis can be implemented efficiently by fast Fourier algorithms. In many common implementations, the calculation of the coefficients in the Chebyshev-Fourier series is performed on a tensor-product polar grid \cite{BoydYu2011,Fornberg1995,Shen2011,Trefethen2000,TownsendWilberWright2017}. The rhodonea nodes used in this work allow a similar computation of the Chebyshev-Fourier coefficients with equivalent efficiency and convergence rates. Compared to the tensor-product case, the new scheme provides the following additional features:
\begin{itemize}[nosep,leftmargin=1em,labelwidth=*,align=left]
\item[-] The data can be collected by sampling along one or several rhodonea curves. This is particularly interesting for the applications in which rose curves are used as scanning trajectories. 
In this perspective, rhodonea nodes can be interpreted as polar analogs of rank-$1$ trigonometric lattices \cite{KKP2012,KPV2015} or rank-$1$ Chebyshev lattices \cite{CoolsPoppe2011,PottsVolkmer2015}.
\item[-] The presented interpolation scheme on the rhodonea curve is more flexible in terms of the underlying interpolation space. The unisolvence of the interpolation problem is guaranteed for a large class of spectral index sets. This is a polar version of a bivariate result in which a similar flexibility is known for polynomial interpolation on
interlacing grids \cite{Floater2017}.  
\end{itemize}
\noindent\emph{Complementation of work on Lissajous nodes.} Rhodonea curves can be regarded as polar counterparts of bivariate Lissajous curves on the square $[-1,1]^2$ and of spherical Lissajous curves. This article is a continuation of the work on polynomial interpolation on Lissajous curves \cite{DenckerErb2017a,DenckerErb2015a,Erb2015,ErbKaethnerAhlborgBuzug2015} and on spherical Lissajous nodes \cite{ErbSphere2017} and extends it to the polar setting. The differences between
the actual work on the disk and the previous works on the hypercube and the unit sphere arise naturally from the differing geometries. In all three settings, the generating curves and the interpolation nodes have own characteristic properties and the interpolation spaces have to be set up according to the given symmetries. Nevertheless, the core ideas in all three setups are similar and many of the ideas used for Lissajous curves can be carried over to the setting of rhodonea curves. In particular, as for multivariate Lissajous-Chebyshev points in the hypercube \cite{DenckerErb2017a,DenckerErb2015a}, a main step in the proof of the quadrature and interpolation formulas is a discrete orthogonality structure linked to the structure of the rhodonea nodes. Compared to previous works, a major progress in this article is the larger flexibility in the choice of the interpolation space. 
 
\vspace{1mm}

{\noindent \bfseries 1.3. Organization.} 
After a short introduction, we provide three different characterizations of the rhodonea nodes: 1) by time equidistant samples along the rhodonea curve (Section \ref{sec:rhodonea}), 2) in terms of a union of two interlacing polar grids (Section \ref{sec:nodesphere}), and 3) by using the algebraic description of the rhodonea varieties (Section \ref{sec:rhodvar}). 

The technical background for the interpolation results in 
form of a discrete orthogonal structure and spectral index sets is given in Section \ref{1507091240}. The main results providing the unisolvence of the spectral interpolation 
on the rhodonea nodes are proven in Section \ref{sec:interpolation}. 

In Section \ref{sec:implementation}, we describe an efficient implementation of the interpolation scheme using the fast Fourier transform. We conclude this work with Section \ref{sec:convergence} and a mathematical description of various properties of the interpolation scheme including: 1) the behavior of the interpolant at the center of the disk, 2) the numerical condition of the scheme, 3) convergence rates, and 4) the application to a Clenshaw-Curtis quadrature rule on the disk. The proofs of all results are collected in Section \ref{sec:proof}. 

\section{Rhodonea curves on the unit disk} \label{sec:rhodonea}

\begin{figure}[htb]
 	\centering
 	\subfigure[\hspace*{1em} The curve $\vect{\rho}^{(2,3)}_{0}$ and the nodes $\LS^{(2,3)}_{0}$.
 	]{\includegraphics[scale=0.9]{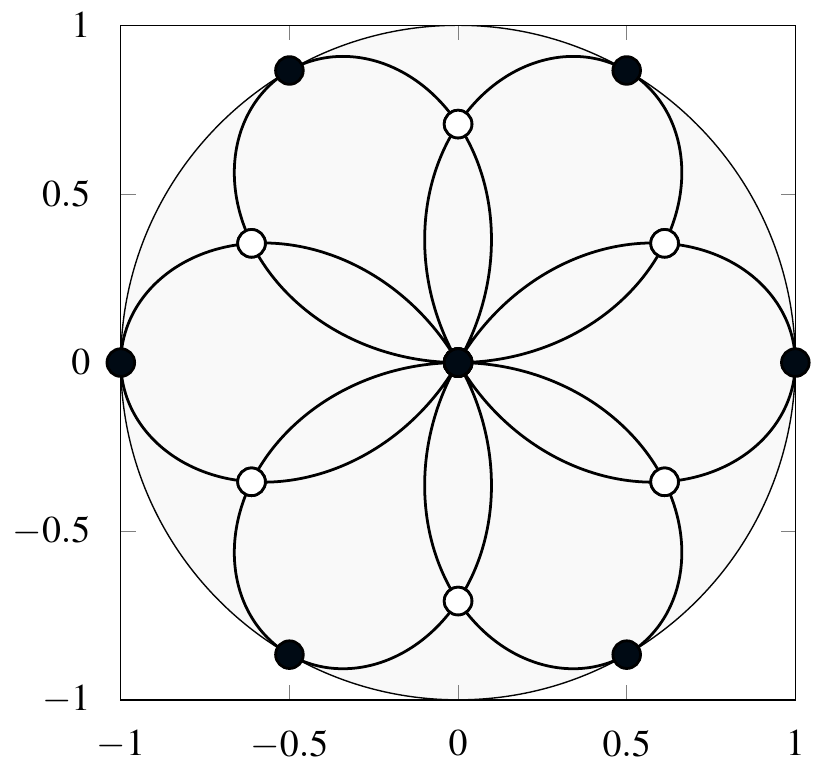}}
 	\hfill	
 	\subfigure[\hspace*{1em} The curve $\vect{\rho}^{(5,3)}_{0}$ and the nodes $\LS^{(5,3)}_{0}$.
 	]{\includegraphics[scale=0.9]{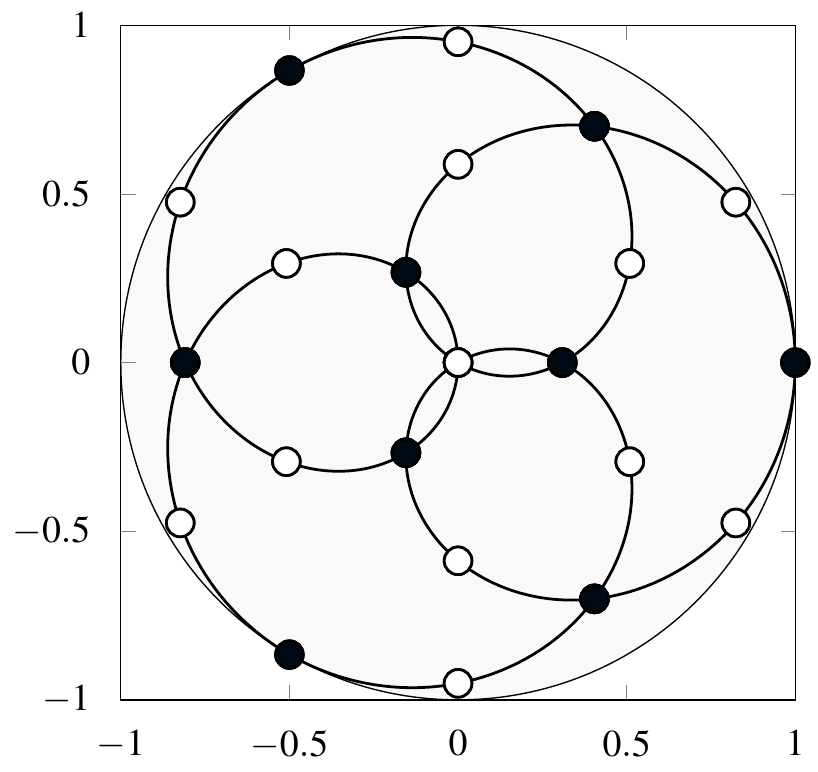}}
   	\caption{Two rhodonea curves $\vect{\varrho}^{(\vect{m})}_0$ and the corresponding nodes $\LSm_0$ as defined in \eqref{201509161237} and \eqref{1709171731}. 
    If $m_1$ and $m_2$ are relatively prime and $m_1 + m_2$ is odd, the set $\LSm_0$ is the union of intersection and boundary points of $\vect{\varrho}^{(\vect{m})}_0$ as described 
   	in Corollary \ref{cor-1} i). In this case, the black and white dots indicate the two interlacing grids determining $\LSm_0$, see Example \ref{ex:1}. If $m_1$ and $m_2$ are both odd only the black samples and the center $(0,0)$ are possible intersection or boundary points of the curve. 
   	} \label{fig:LS-1}
 \end{figure}

\paragraph{\bf 2.1. General properties} For a frequency vector $\vect{m} = (m_1, m_2) \in\Nn^{2}$ and a rotation parameter $\alpha \in \Rr$, we define the \emph{rhodonea curves} in parametric form as
\begin{equation} \label{201509161237}
\rhodonea(t) = \Big( \cos(m_2 t) \cos (m_1 t - \alpha \pi), \ \cos(m_2 t) \sin (m_1 t - \ts \alpha \pi) \Big), \quad t \in \mathbb{R}.
\end{equation}
The rhodonea curve $\rhodonea$ is contained in the unit disk $\Dd = \{\vect{x} \in \Rr^2 \ : \ |\vect{x}| \leq 1\}$. Because of its characteristic shape of a patelled rose, these curves are also referred to as \emph{rose curves} or 
\emph{roses of Grandi}, after the monk and mathematician Guido Grandi who studied them intensively in \cite{Grandi1728}. Two typical examples of rose curves are illustrated in Figure \ref{fig:LS-1}.

The frequency parameters $m_1$ and $m_2$ in the curve $\rhodonea$ determine a superposition of 
a radial and an angular harmonic motion. For this reason, rose curves can also be regarded 
as polar variants of bivariate Lissajous curves \cite{DenckerErb2017a,DenckerErb2015a,Erb2015,ErbKaethnerAhlborgBuzug2015}. If the numbers
$m_1$ and $m_2$ are relatively prime, the minimal period $P$ of $\rhodonea$ is given by $P=2\pi$ if $m_1+m_2$ is odd, and $P = \pi$ if $m_1+m_2$ is even (see Proposition \ref{prop-11}). Depending on these two cases, the properties of the curve $\rhodonea$ vary slightly and we will have to distinguish them at several occasions.  

For general $\vect{m} \in \Nn^2$ and $g = \mathrm{gcd}(\vect{m}) \geq 1$,
we can write $\rhodonea(t) = \vect{\varrho}^{(\vect{m}/g)}_{\alpha}( g t)$. In this case,
the minimal period of $\rhodonea$ is given by $P/g$. In particular, all properties of a rose $\rhodonea$ with general $\vect{m} \in \Nn^2$ can be obtained from the 
curve $\vect{\varrho}^{(\vect{m}/g)}_{\alpha}$ with the relatively prime parameter $\vect{m}/g$. When analyzing the properties of a single rose curve it is therefore enough to restrict the considerations to relatively prime frequency numbers $m_1$ and $m_2$. However, if more than one rhodonea curve is used to generate the interpolation nodes, also the general case will play an important role later on.

\paragraph{\bf 2.2. The self-intersection points of $\rhodonea$} To extract all self-intersection points of the curve $\rhodonea$, we consider for $t \in [0,2\pi)$
the sets $\mathcal{S}^{(\vect{m})}(t) = \{ s \in [0,2\pi): \ \rhodonea(s) = \rhodonea(t) \}$ and the sampling points
\begin{align} t^{(\vect{m})}_{l} &= \frac{l \pi}{2 m_1 m_2}, \quad l \in \{0,1, \ldots, 4m_1m_2-1\}. \label{eq-samples1}
\end{align}

\begin{proposition} \label{prop-11} Let $m_1$ and $m_2$ be relatively prime numbers.\\[-2mm]

\noindent If $m_1+m_2$ is odd, the minimal period of $\rhodonea$ is $2 \pi$ and
 \[  \begin{array}{lll}
      (i) & \# \mathcal{S}^{(\vect{m})}(t) = 2 m_2  & \text{if}\quad t \in \{\,t^{(\vect{m})}_{l}\,| \ l\in \{0,\ldots,4m_1m_2-1\}, \ l \equiv m_1 \mod 2 m_1 \},\\
      (ii) & \# \mathcal{S}^{(\vect{m})}(t) = 2 & \text{if}\quad t \in \{\,t^{(\vect{m})}_{l}\,| \ l\in \{0,\ldots,4m_1m_2-1\}, \ l \not\equiv 0 \mod m_1 \}, \\
      (iii) & \# \mathcal{S}^{(\vect{m})}(t) = 1 & \text{for all other $t \in [0,2\pi)$.}
       \end{array}
\]
If $m_1+m_2$ is even, then the minimal period of $\rhodonea$ is $\pi$ and
 \[  \begin{array}{lll}
      (i)' & \# \mathcal{S}^{(\vect{m})}(t) = 2 m_2  & \text{if}\quad t \in \{\,t^{(\vect{m})}_{l}\,| \ l\in \{0,\ldots,4m_1m_2-1\}, \ l \equiv m_1 \mod 2 m_1 \}, \\
      (ii)' & \# \mathcal{S}^{(\vect{m})}(t) = 4 & \text{if}\quad t \in \{\,t^{(\vect{m})}_{2l}\,| \ l\in \{0,\ldots,2m_1m_2-1\}, \ l \not\equiv 0 \mod m_1 \}, \\
      (iii)' & \# \mathcal{S}^{(\vect{m})}(t) = 2 & \text{for all other $t \in [0,2\pi)$.}
      \end{array}
\]
\end{proposition}

We can extract a series of properties from this result. The nodes $\rhodonea(t)$ in $(i)$ and $(i)'$ with $\# \mathcal{S}^{(\vect{m})}(t) = 2 m_2$ correspond to the center $(0,0)$ of the unit disk $\Dd$. As $t$ varies from $0$ to $P$, the center is traversed $2m_2$ 
times in the case that $m_1+m_2$ is odd and $m_2$ times if $m_1+m_2$ is even. All the points
$\rhodonea(t)$ in $(ii)$ and $(ii)'$ are doubly traversed in one period $P$. Therefore, if $m_1+m_2$ is odd, Proposition \ref{prop-11} ensures that the set
\begin{equation} \label{1709171731}
 \LSm_{\alpha} = \left\{\,\rhodonea(t^{(\vect{m})}_{l})\,|\,  l\in \{0,\ldots,4m_1m_2-1\} \,\right\}
\end{equation}
contains all self-intersection points of the curve $\rhodonea$. The additional nodes
$\rhodonea(t^{(\vect{m})}_{l})$ with $l \equiv 0 \mod 2 m_1$ describe precisely the set of all points at which the curve $\rhodonea$ touches the boundary of the unit disk $\Dd$ (i.e. the unit circle). If $m_1+m_2$ is even, 
the set $\LSm_{\alpha}$ is larger than the union of self-intersection and boundary points of $\rhodonea$. Nevertheless, also in this case the set $\LSm_{\alpha}$
will play an important role in our considerations. We summarize all important properties of the rhodonea curves in the following Corollary \ref{cor-1}.

\begin{corollary} \label{cor-1}
Let $m_1$ and $m_2$ be relatively prime natural numbers. 
\begin{itemize}
\item[i)] If $m_1 + m_2$ is odd, then $\LSm_\alpha$ is the union of all self-intersection and all boundary points of the closed curve
$\rhodonea$. $\LSm_\alpha$ contains
$2 m_1 m_2 + 1$ points in $\Dd$. It includes the center $(0,0)$ that is traversed $2 m_2$ times in one period $P = 2\pi$, $2 (m_1-1) m_2$ ordinary double points distinct
from $(0,0)$ and $2m_2$ points on the boundary of $\Dd$.
\item[ii)] If $m_1 + m_2$ is even, the curve $\rhodonea$ contains $\frac12(m_1-1) m_2$ ordinary double points distinct from $(0,0)$ and $m_2$ points on the boundary of $\Dd$. The
center $(0,0)$ is traversed $m_2$ times in one period $P = \pi$. 
\end{itemize}
\end{corollary}

\begin{remark} Various properties of the rhodonea curves described in this section are known for a long time. The number and type of intersection points are, for instance, originally derived in \cite{Himstedt1888}. A general historic overview with a lot of additional features of 
rhodonea curves can be found in \cite[p. 297-306]{Loria1902}. Further graphical illustrations
of rhodonea curves are given in \cite{Gorjanc2010}. The novel aspects of this article are the different characterizations of the rhodonea nodes $\LSm_\alpha$. In Corollary \ref{cor-1} i) we could describe the union of intersection and boundary points as the set $\LSm_\alpha$ of
time-equidistant samples along the rhodonea curve. Further characterizations of $\LSm_\alpha$ are now obtained in the next part.
\end{remark}

\section{The interpolation nodes generated by rhodonea curves} \label{sec:nodesphere}

\begin{figure}[htb]
 	\centering
 	\subfigure[The nodal index set $\I^{(5,3)}$. The black indices at $i_1 = 0$ describe the nodes 
 	of $\LS^{(5,3)}$ at the boundary of $\Dd$. The white indices at $i_1 = 5$ are mapped onto
 	the center $(0,0)$.  
 	]{\includegraphics[scale=0.9]{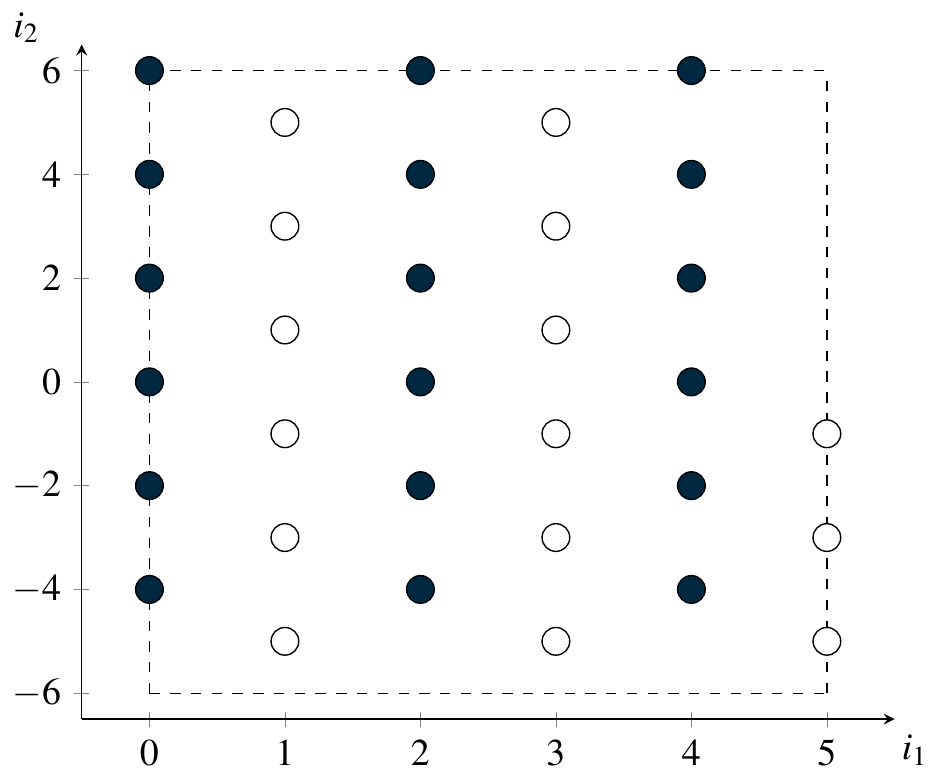}}
 	\hfill	
 	\subfigure[The rhodonea nodes $\LS^{(5,3)}$ and the variety $ \mathcal{R}^{(5,3)}=\bigcup_{\rho = 0}^{1 }\vect{\rho}^{(5,3)}_{\rho/3}$. The red curve displays the curve $\vect{\rho}^{(5,3)}_{0}$. We have $\LS^{(5,3)}_0 \subset \LS^{(5,3)}$.
 	]{\includegraphics[scale=0.9]{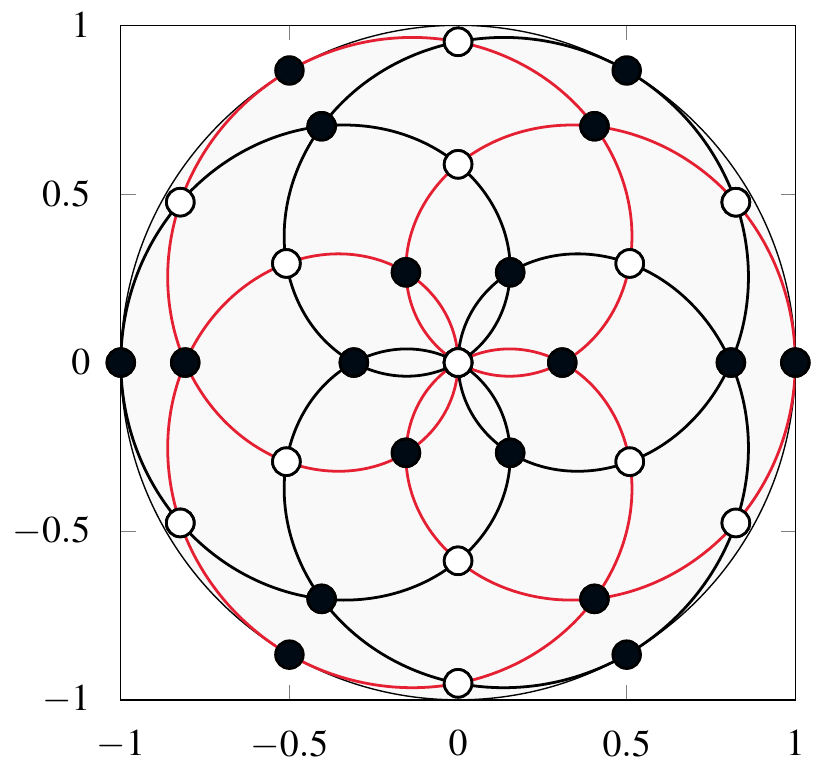}}
   	\caption{Illustration of the nodal index set $\I^{(5,3)}$, the rhodonea nodes 
   	$\LS^{(5,3)}$ and rhodonea variety $\mathcal{R}^{(5,3)}$. The black and white nodes form two
   	interlacing subgrids and are determined by the subsets $\I^{(5,3)}_0$ and $\I^{(5,3)}_1$, respectively. Compare also with Figure \ref{fig:LS-1} b) where an illustration of the curve $\vect{\rho}^{(5,3)}_{0}$ and $\LS^{(5,3)}_0$ is given.  
   	} \label{fig:LS-2a}
\end{figure}

The nodes $\LSm_{\alpha}$ of the rhodonea curve $\rhodonea$ given in Corollary \ref{cor-1} i) 
can be described as the union of two interlacing rectangular grids in polar coordinates.
Without restriction to generality, we set $\alpha = 0$ and consider the nodes $\LSm_{0}$. 
Further, we will use general frequencies $\vect{m} = (m_1, m_2) \in\Nn^{2}$ for this second description. 
If $m_1$ and $m_2$ are not relatively prime, the so obtained nodes contain $\LSm_{0}$ as a subset and can be interpreted as sampling nodes of more than one rhodonea curve. Similarly, if $m_1+m_2$ is even the given description will contain the nodes $\LSm_{0}$ as a subset. First examples are given in Figure \ref{fig:LS-2a} and \ref{fig:LS-2b}. 

\begin{figure}[htb]
 	\centering
 	\subfigure[The nodal index set $\I^{(4,4)}$. The indices at $i_1 = 0$ describe the nodes 
 	of $\LS^{(4,4)}$ at the boundary of $\Dd$. The indices at $i_1 = 4$ are mapped onto
 	the center $(0,0)$.  
 	]{\includegraphics[scale=0.9]{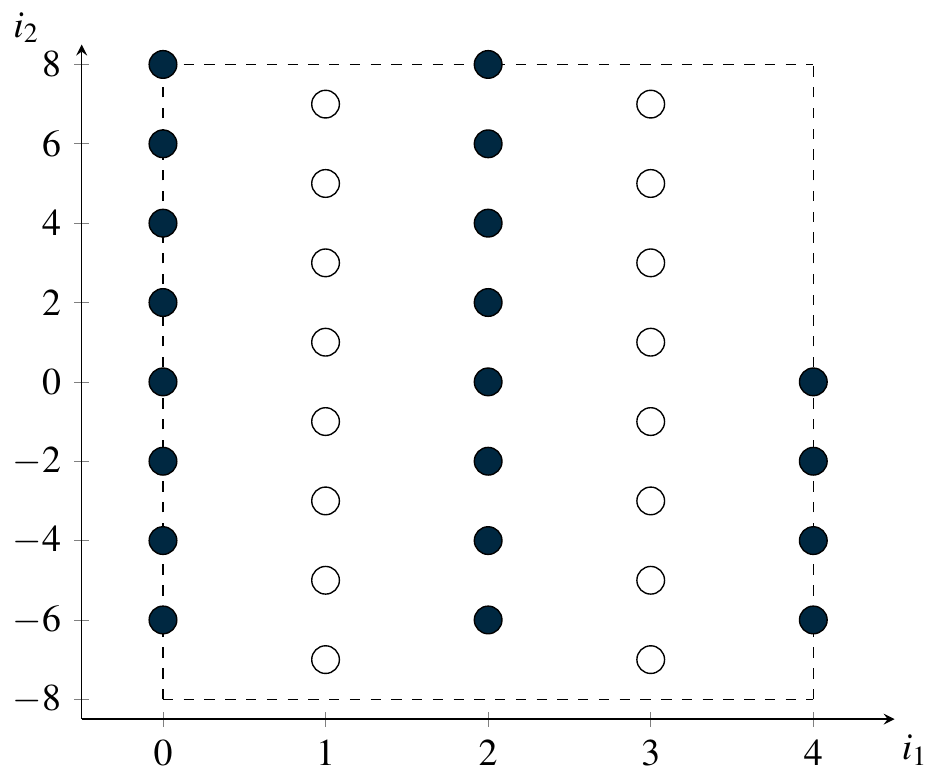}}
 	\hfill	
 	\subfigure[The rhodonea nodes $\LS^{(4,4)}$ and the variety $ \mathcal{R}^{(4,4)} = \bigcup_{\rho = 0}^{7 }\vect{\rho}^{(4,4)}_{\rho/4}$.
 	The red curve illustrates the circle $\vect{\rho}^{(4,4)}_{0}$. We have $\LS^{(4,4)}_0 \subset \LS^{(4,4)}$.
 	]{\includegraphics[scale=0.9]{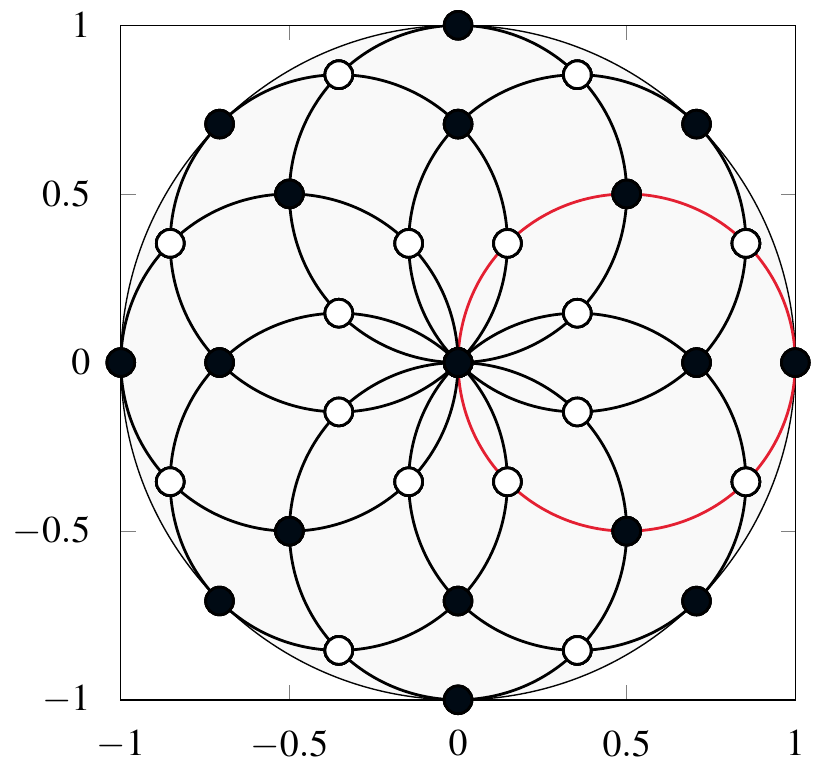}}
   	\caption{Illustration of the index set $\I^{(4,4)}$, the rhodonea nodes 
   	$\LS^{(4,4)}$ and rhodonea variety $\mathcal{R}^{(4,4)}$. The black and white nodes form two
   	interlacing subgrids determined by the subsets $\I^{(4,4)}_0$ and $\I^{(4,4)}_1$, respectively.   
   	} \label{fig:LS-2b}
\end{figure}

To obtain this more general characterization, we introduce the nodal index set
\begin{equation}\label{eq:0911}
\Imm = \left\{ \,(i_1, i_2) \in \Zz^{2}\ \left|\begin {array}{ll}
& 0\leq i_{1}\leq m_{1}, \; -2 m_2 < i_2 \leq 2 m_2, \\
& \text{$i_2 \leq 0$ \, if $i_{1} = m_1$}, \\
& \text{$i_{1} + i_{2}$ is even}
\end{array}\right. \, \right\}.
\end{equation}
The set $\Imm$ can be decomposed into the two disjoint finite grids $\Imm_0$ and $\Imm_1$ given by 
\begin{equation} \label{eq:0901} 
\Imm_0 = \{(i_1, i_2) \in \Imm \ | \ \text{$i_1$, $i_2$ are even} \, \}, \quad
\Imm_1 = \{(i_1, i_2) \in \Imm \ | \ \text{$i_1$, $i_2$ are odd} \, \}.
\end{equation} 
A tuple $\vect{i} = (i_1, i_2)$ in $\Imm$ has a one to one relation to a point in polar coordinates by introducing a radial and an angular component
\[ r^{(m_1)}_{i_1} = \cos \left( \frac{i_1}{2 m_1} \pi \right) \in [0,1], \qquad \theta^{(m_2)}_{i_2} = \frac{i_2}{2 m_2} \pi \in (-\pi,\pi].\]
The general \emph{rhodonea nodes} on the unit disk are then defined as the point set
\begin{equation} \label{eq:09172}
\LSm =\left\{\, \vect{x}^{(\vect{m})}_{\vect{i}}\,\left|\,\vect{i}\in \Imm \right.\right\},
\end{equation}
with the nodes $\vect{x}^{(\vect{m})}_{\vect{i}} \in \Dd$ given in polar coordinates 
$(r^{(m_1)}_{i_1}, \theta^{(m_2)}_{i_2})$ by
\begin{equation} \label{eq:0917234} \vect{x}^{(\vect{m})}_{\vect{i}} = \left(r^{(m_1)}_{i_1} \cos (\theta^{(m_2)}_{i_2}), r^{(m_1)}_{i_1} \sin (\theta^{(m_2)}_{i_2})\right).\end{equation}
From the almost rectangular form of $\Imm_0$ and $\Imm_1$ in \eqref{eq:0901} (see also Figure \ref{fig:LS-2a} (a) and \ref{fig:LS-2b} (a)), the cardinalities $\# \Imm_0$ and $\# \Imm_1$ can be determined by a simple counting argument:
\begin{equation}\label{eq:0917}
\#\Imm_0 = (m_1+1) m_2, \quad \#\Imm_1 = m_1 m_2, \quad \#\Imm =\#\Imm_0 +\#\Imm_1 = (2 m_1+1) m_2.
\end{equation}
Since the $m_2$ points $\vect{x}^{(\vect{m})}_{\vect{i}}$ with coordinate $i_1 = m_1$ all describe the center $(0,0)$ of $\Dd$, the cardinality
of $\LSm$ is smaller campared to $\#\Imm$:
\[\#\LSm = 2 m_1 m_2 + 1.\] 
In the setting of Corollary \ref{cor-1} i), the cardinality of the node set $\LSm_0$ corresponds exactly to the cardinality of the set $\LSm$. For general $\vect{m} \in \Nn^{2}$ with $g = \gcd(\vect{m}) \geq 1$, we have the following relation between $\LSm$ and the node points 
$\LSm_{\alpha}$ defined in \eqref{1709171731}. 

\begin{theorem} \label{cor-111}
Let $\vect{m} \in \Nn^2$. Then
\[\LSm = \bigcup_{\rho = 0}^{2 g-1} \LSm_{\rho/m_2}  = 
\bigcup_{\rho = 0}^{2 g-1}\left\{\,\vect{\rho}^{(\vect{m})}_{\rho/m_2} (t^{(\vect{m})}_{l})\ \Big| \  l\in \{0,\ldots,4m_1m_2/g-1\} \,\right\},\]
where $\rhodonea$ are the rhodonea curves introduced in \eqref{201509161237}, $t^{(\vect{m})}_{l}$ the equidistant sampling points given in \eqref{eq-samples1}, and 
$\LSm_{\alpha}$ the node points defined in \eqref{1709171731}.
\end{theorem}

\begin{example} \label{ex:1}
If $m_1$ and $m_2$ are relatively prime, we have $g = \gcd (\vect{m}) = 1$ and Theorem \ref{cor-111} states that $\LSm$ is generated by the two rhodonea curves
$\vect{\rho}^{(\vect{m})}_{0}$ and $\vect{\rho}^{(\vect{m})}_{1/m_2}$. Let $a$ and $b$ be two integers from B\'ezout's lemma such that $a m_1 + b m_2 = 1$. Then
\begin{equation}
\label{eq:111123421324}
\vect{\rho}^{(\vect{m})}_{0}(t - a \pi / m_2) = (-1)^{a+b} \vect{\rho}^{(\vect{m})}_{1/m_2}(t).
\end{equation}

\begin{enumerate}[label=\roman*)]
 \item
If $m_1+m_2$ is odd, then the rhodonea curve $\vect{\rho}^{(\vect{m})}_{0}$ satisfies $\vect{\rho}^{(\vect{m})}_{0}(t) = -\vect{\rho}^{(\vect{m})}_{0}(t-\pi)$ and $\vect{\rho}^{(\vect{m})}_{0}(\Rr)$ is point symmetric with respect
to the origin. The same holds true for the nodes $\LSm_0$ given in \eqref{1709171731}. Therefore,
the identity \eqref{eq:111123421324} implies
that $\vect{\rho}^{(\vect{m})}_{0}(\Rr) = \vect{\rho}^{(\vect{m})}_{1/m_2}(\Rr)$ and $$\LSm = \LSm_0 = \LSm_{1/m_2}.$$ Thus, according to Corollary \ref{cor-1} i), the set $\LSm$ corresponds to the union of self-intersection and boundary points of the curve $\vect{\rho}^{(\vect{m})}_{0}$. Rotating the set $\LSm$ by an angle $\alpha \pi$ gives a corresponding identity for 
the points $\LSm_\alpha$.
\item If $m_1+m_2$ is even, then $a + b$ is odd and \eqref{eq:111123421324} gives
$\vect{\rho}^{(\vect{m})}_{0}(t - a \pi / m_2) = - \vect{\rho}^{(\vect{m})}_{1/m_2}(t)$. Further, in this case the sets $\vect{\rho}^{(\vect{m})}_{0}(\Rr)$ and $\LSm_0$ are not
point symmetric with respect to the center $(0,0)$. This implies that $\LSm_{1/m_2} = -\LSm_0 \neq \LSm_0$ and $$\LSm = \LSm_0 \cup -\LSm_0.$$
In particular, $\LSm$ is generated by the samples of two distinct rhodonea curves. 
\end{enumerate}
\end{example}

\section{A link to rhodonea varieties} \label{sec:rhodvar}

The union of rhodenea curves used to generate the nodes $\LSm$ in Theorem \ref{cor-111} can be identified as an algebraic variety. For $r \in [-1,1]$, we denote by 
$T_{m_1}(r) = \cos (m_1 \arccos r)$ the univariate Chebyshev polynomial of degree $m_1$ and by $H_{m_2}(x_1,x_2)$ the bivariate polynomial 
\[ H_{m_2}(x_1,x_2) = \sum_{k=0}^{\lfloor m_2/2 \rfloor} \binom{m_2}{2k} (-1)^k x_1^{m_2-2k} x_2^{2k}.\]
$H_{m_2}$ is a bivariate homogeneous polynomial of total degree $m_2$. The \emph{rhodonea variety} $\mathcal{R}^{\vect{m}}$ on the unit disk $\Dd$ is defined as
\begin{equation} \label{1509222010} 
\mathcal{R}^{(\vect{m})}
= \left\{\,\vect{x} \in \Dd \,\Big| \, (x_1^2+x_2^2)^{m_2} T_{m_1}\left(\sqrt{x_1^2+x_2^2}\right)^2 = H_{m_2}(x_1,x_2)^2 \, \right\}.
\end{equation}
This affine real algebraic variety is of order $2 m_1 + 2 m_2$. In polar coordinates, we get a simpler description of this variety. With the substitution $x_1(r,\theta) = r \cos \theta$ and
$x_2(r,\theta) = r \sin \theta$ and the trigonometric formula 
\[ \cos (m_2 \theta) = \sum_{k=0}^{\lfloor m_2/2 \rfloor} \binom{m_2}{2k} (-1)^k \cos(\theta)^{m_2-2k} \sin(\theta)^{2k} \]
we can rewrite \eqref{1509222010} as
\begin{equation} \label{1509222011} 
\mathcal{R}^{(\vect{m})}
= \left\{\,\left.\vect{x}(r,\theta) \in \Dd \,\right|\, T_{m_1}(r)^2 = \cos^2(m_2 \theta) \,\right\}.
\end{equation}
Since $H_{m_2}(\cos \theta, \sin \theta) = \cos (m_2 \theta)$ we see that $H_{m_2}(x_1,x_2)$ is in fact a harmonic homogeneous polynomial of degree $m_2$. 

\begin{theorem} \label{thm:decompositionrhodonea}

\begin{enumerate}[label=\alph*)]
\item The variety $\mathcal{R}^{(\vect{m})}$ can be decomposed as
$\displaystyle\mathcal{R}^{(\vect{m})} = \bigcup_{\rho = 0}^{2 g-1} \vect{\rho}^{(\vect{m})}_{\rho/m_2}([0,P)).$
\item The rhodonea nodes $\LSm$ can be written as
\[\LSm = \left\{\,\left.\vect{x}(r,\theta) \in \Dd \,\right|\, T_{m_1}(r)^2 = \cos^2(m_2 \theta) \in \{0,1\} \,\right\}, \]
i.e., the set $\LSm$ consists of those points of the variety $\mathcal{R}^{(\vect{m})}$ for which $T_{m_1}(r)^2$ and $\cos^2(m_2 \theta)$ get maximal or minimal. 
\end{enumerate}
\end{theorem}

\begin{remark}
We mention (without explicit proof) that in addition to the statements of Theorem \ref{thm:decompositionrhodonea}, the points in $\LSm$ can also be categorized in terms of singularity theory. 
This yields a description similar to the one given in Corollary \ref{cor-1} i). Namely, the elements of $\LSm$ in the interior of the unit disk $\Dd$ are precisely 
the singular points of the algebraic variety $\mathcal{R}^{(\vect{m})}$. The singular points distinct from the center $(0,0)$ are all ordinary double points while the center itself is a singular
point with multiplicity $2 m_2$. 
\end{remark}

\begin{example} 
\begin{enumerate}[label=\roman*)]
 \item We consider the setting of Corollary \ref{cor-1} i), i.e., $m_1$ and $m_2$ are relatively prime and $m_1 + m_2$ is odd. Then, according to Example 
 \ref{ex:1} i) and Theorem \ref{thm:decompositionrhodonea} a), we have $\mathcal{R}^{(\vect{m})} = \vect{\rho}^{(\vect{m})}_{0}([0,2 \pi))$, i.e. $\mathcal{R}^{(\vect{m})}$ provides
 the algebraic equation of the curve $\vect{\rho}^{(\vect{m})}_{0}$ given in parametric form in \eqref{201509161237}. This characterization of $\vect{\rho}^{(\vect{m})}_{0}$ is well known in the literature. Slightly less compact variants of the definition in \eqref{1509222010} can be found in 
 \cite{Gorjanc2010,Himstedt1888,Loria1902}. 
  \item If $m_1$ and $m_2$ are relatively prime and $m_1+m_2$ is even, we get a different scenario. In this case, Example 
 \ref{ex:1} ii) and Theorem \ref{thm:decompositionrhodonea} a) imply that 
 $\mathcal{R}^{(\vect{m})} = \vect{\rho}^{(\vect{m})}_{0}([0,\pi)) \cup -\vect{\rho}^{(\vect{m})}_{0}([0,\pi))$, i.e. the algebraic variety 
 $\mathcal{R}^{(\vect{m})}$ is the union of two distinct rhodonea curves. The algebraic varieties $\mathcal{R}^{(\vect{m})}_+$ and $\mathcal{R}^{(\vect{m})}_-$ describing 
 the single curves $\pm \vect{\rho}^{(\vect{m})}_{0}$ are given as
 \[ \mathcal{R}^{(\vect{m})}_{\pm} = \left\{\,\left.\vect{x} \in \Dd \,\right|\, (x_1^2+x_2^2)^{\frac{m_2}{2}} 
 T_{m_1}\left(\sqrt{x_1^2+x_2^2}\right) = \pm H_{m_2}(x_1,x_2) \,\right\}.\]
 $\mathcal{R}^{(\vect{m})}_{\pm}$ are algebraic varieties only in the given particular case that both $m_1$ and $m_2$ are odd.
 The description of $\vect{\rho}^{(\vect{m})}_{0}$ as the algebraic variety $\mathcal{R}^{(\vect{m})}_+$  is also usually provided in the literature, see \cite{Gorjanc2010,Himstedt1888,Loria1902}. 
 The particular variety $\mathcal{R}^{(5,3)}$ and its two subvarieties $\mathcal{R}^{(5,3)}_{\pm}$ are illustrated in Figure \ref{fig:LS-2a} (b). 
 \item For $\vect{m} = (1,1)$, the points of the rhodonea variety $\mathcal{R}^{(1,1)}$ satisfy the equation 
 \[(x_1^2 + x_2^2)^2 = x_1^2 \quad \Leftrightarrow \quad x_1^2 + x_2^2 = \pm x_1.\]
 In this case, the variety consists of two circles with diameter $1$ and radius $1/2$ centered at $\pm 1/2$. The two circles correspond to the
 two rhodonea curves $\vect{\varrho}^{(1,1)}_0$ and $\vect{\varrho}^{(1,1)}_1$. The variety $\mathcal{R}^{(1,1)}$ is part of the larger 
 variety $\mathcal{R}^{(4,4)}$ illustrated in Figure \ref{fig:LS-2b} (b).
  \item For $\vect{m} = (1,2)$, the rhodonea variety $\mathcal{R}^{(1,2)}$ is determined by the equation 
 \[(x_1^2 + x_2^2)^3 = (x_1^2-x_2^2)^2.\]
 The corresponding curve $\vect{\varrho}^{(1,2)}_0$ gives the so called four leave rose, a curve having the form of a rose with four petals. In general, the curve $\vect{\varrho}^{(1,m_2)}_0$
 has the shape of a rose with $2m_2$ leaves if $m_2$ is even, and $m_2$ leaves if $m_2$ is odd. The description of $\vect{\varrho}^{(1,m_2)}_0$ in polar coordinates is given by $r = \cos (m_2 \theta)$.
For some illustrations of these roses we refer to \cite{Gorjanc2010}.
\end{enumerate}

\end{example}

\section{Spectral index sets and discrete orthogonality on $\Imm$}\label{1507091240}

\paragraph{\bf 5.1. Discrete function space on $\Imm$}

We denote by $\mathcal{L}(\Imm)$ the space of all discrete functions $f: \Imm \to \Cc$ 
on the index set $\Imm$. 
In $\mathcal{L}(\Imm)$, we consider further the family of functions $\polbas\in \mathcal{L}(\Imm)$, $\vect{\gamma} \in \Zz^2$, given by
\begin{equation}\label{A1508291531}
\polbas(\vect{i}) = \cos(\gamma_{1} i_1 \pi/(2m_{1})) \mathrm{e}^{\imath \gamma_{2} i_2 \pi/(2m_{2}) }.
\end{equation}
In the following, our objective is to derive a discrete orthogonality structure 
for the functions $\polbas$ on $\Imm$. As the functions $\polbas$ are a discretization of 
the Chebyshev-Fourier basis $\polbascont$ AT the nodes $\LSm$ (this will be derived in Section \ref{sec:interpolation}), this discrete orthogonality is the key ingredient for 
the proof of the main Theorems \ref{201512131945} and \ref{201512131946} on 
spectral interpolation on the rhodonea nodes. 

To introduce an inner product on the space $\mathcal{L}(\Imm)$ we define for $\vect{i}\in \Imm$ the weights
\begin{equation}\label{1507091748}
\mathrm{w}^{(\vect{m})}_{\vect{i}}= \frac{1}{4m_1m_2}\left\{ \begin{array}{ll} 1 & \text{if}\ \ \vect{i}\in\Imm, \; i_1 = 0\\
                                                                2     & \text{if}\ \ \vect{i}\in\Imm, \; 0 < i_1 \leq m_1.
                                                                \end{array} \right.
\end{equation}
The corresponding discrete measure $\mathrm{w}$ on the power set of $\Imm$ is defined by
$\mathrm{w}(\{\vect{i}\})=\mathrm{w}^{(\vect{m})}_{\vect{i}}$. Then, the inner product
\[ \langle f,g \rangle_{\mathrm{w}} = \int f \, \overline{g} \, \mathrm{d}\rule{1pt}{0pt}\mathrm{w}
= \sum_{\vect{i} \in \Imm} \mathrm{w}^{(\vect{m})}_{\vect{i}} f(\vect{i}) \overline{g(\vect{i})}\]
turns $\mathcal{L}(\Imm)$ into a Hilbert space. 
We denote the corresponding norm by $\|\cdot\|_{\mathrm{w}}$.
 
\begin{definition} \label{def:spectralindex} We call $\Gam \subset \Zz^2$ a \emph{spectral index set for $\Imm$} if the system $\{\polbas \ | \ \vect{\gamma} \in \Gam\}$
forms an orthogonal basis of the inner product space $(\mathcal{L}(\Imm),\langle\,\cdot,\cdot\,\rangle_{\mathrm{w}})$.
We additionally assume that the spectral index set $\Gam$ is a subset of
\[ \Kmm = \{ \vect{\gamma} \in \Zz^2 \ | \ 0 \leq \gamma_1 \leq 2 m_1, \ -2 m_2 < \gamma_2 \leq 2 m_2 \}. \]
and that $\gamma_1 + \gamma_2$ is even for all $\vect{\gamma} \in \Gam$. The last condition is referred to as parity condition. 
\end{definition}

 \begin{figure}[htb]
 	\centering
 	\subfigure[The triangular spectral index set $\vect{\Gamma}^{(5,3)}_{\triangle}$.
 	]{\includegraphics[scale=0.95]{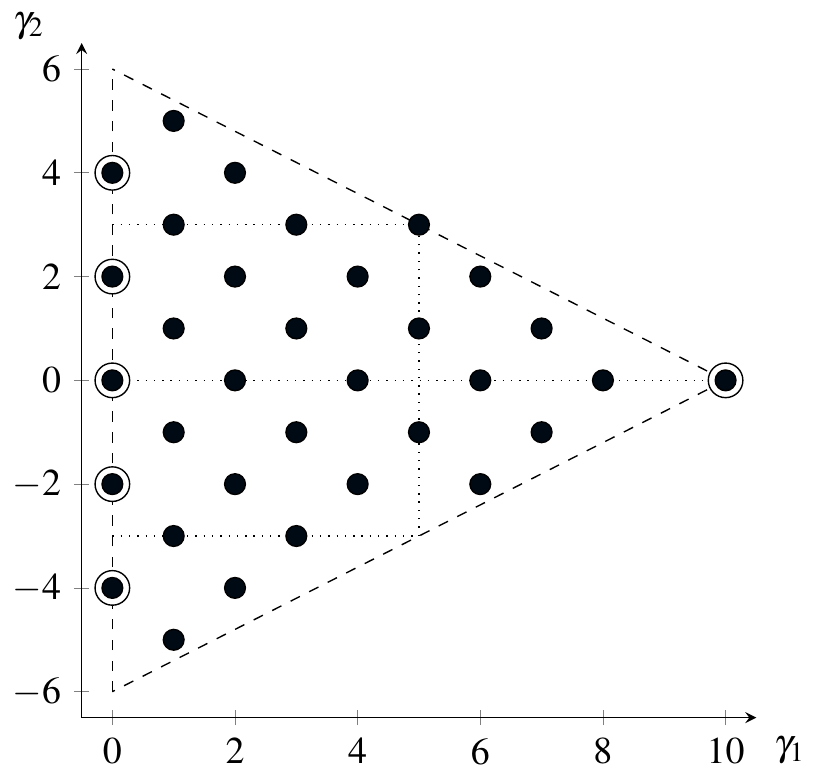}}
 	\hfill	
 	\subfigure[The rectangular spectral index set $\vect{\Gamma}^{(5,3)}_{\square}$.
 	]{\includegraphics[scale=0.95]{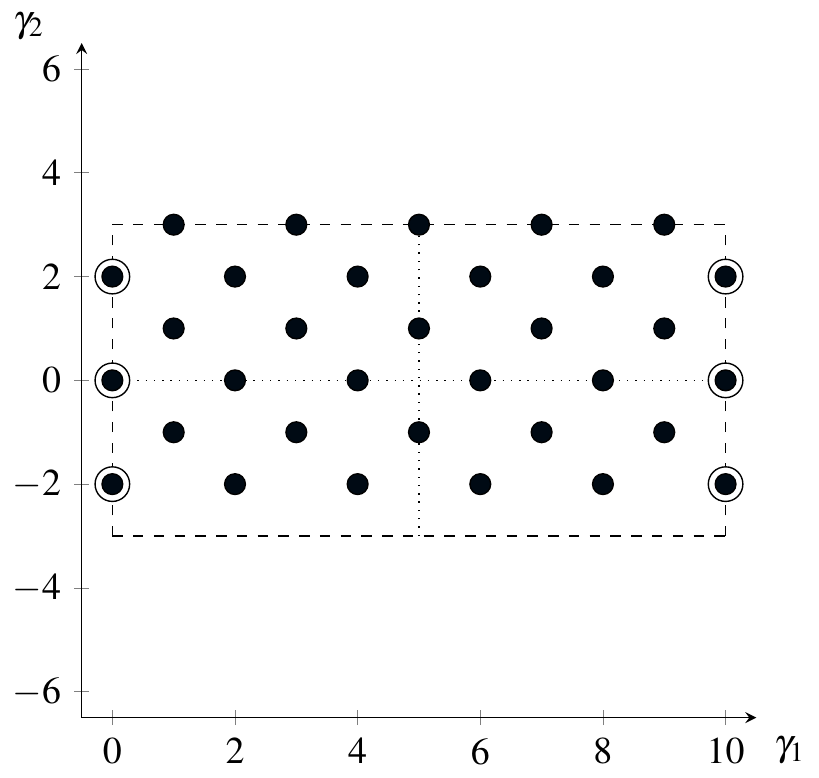}}
   	\caption{Illustration of the spectral index sets $\Gamtriangle$ and $\Gamsquare$.
   	The circled dots indicate the basis functions in \eqref{1508221825} and \eqref{1508221825BB} with norm $1$.
   	} \label{fig:LS-3}
 \end{figure}

\paragraph{\bf 5.2. Rectangular spectral index sets}
For our purpose, the most important example of a spectral index set is the rectangular set
\begin{equation}\label{1508222042base}
\Gamsquare = \left\{\,\vect{\gamma}\in \Kmm \left| \ - m_2 < \gamma_{2} \leq m_{2}, \ 
\gamma_1 + \gamma_2 \ \text{is even}
\right.
\right\}.
\end{equation}
The set $\Gamsquare$ contains $\# \Gamsquare = (2m_1 + 1) m_2$ elements. This corresponds exactly to the cardinality of $\Imm$. In fact, we obtain: 

\begin{theorem}\label{1507091911} The set $\Gamsquare$ is a spectral index set for $\Imm$, i.e. 
the system $\{ \polbas \ | \ \vect{\gamma}\in \Gamsquare \}$ is an
orthogonal basis of the $(2m_1+1) m_2$ dimensional space $(\mathcal{L}(\Imm),\langle\,\cdot,\cdot\,\rangle_{\mathrm{w}})$. 
The basis functions $\polbas$ are normed by 
\begin{equation}\label{1508221825}
 \|\polbas\|_{\mathrm{w}}^2 =
 \left\{ \begin{array}{rl} 1,\; & \text{if}\quad \vect{\gamma} \in \Gamsquare, \ \gamma_1 \in \{0,2m_1\},\\
 \frac12,\; & \text{if}\quad \vect{\gamma} \in \Gamsquare, \ \gamma_1 \notin \{0,2m_1\}.
   \end{array} \right.
\end{equation}
\end{theorem}

\paragraph{\bf 5.3. General spectral index sets} Based on the rectangular index set $\Gamsquare$, we can characterize all further spectral index sets $\Gam$ contained in Definition \ref{def:spectralindex}. 
On $\Kmm$, we define a flip operator ${}^*: \Kmm \to \Kmm$, $\vect{\gamma} \to 
\vect{\gamma}^*$ by
\begin{equation} \label{eq:201807201118} \vect{\gamma}^* = (2 m_1 - \gamma_1 , \gamma_2 + 2m_2 \mod 4 m_2).\end{equation}
This flip operator combines a reflection at $\gamma_1 = m_1$ with a glide operation along the $\gamma_2$ coordinate. If $\gamma_1 +\gamma_2$ is even,
then the basis functions $\polbas$ are invariant under this glide-reflection operation, i.e we have for all $\vect{i} \in \Imm$:
\begin{equation} \label{eq:glidereflection}
 \chi_{\vect{\gamma}^*}^{(\vect{m})}(\vect{i}) = \polbas(\vect{i}) \quad \text{if $\gamma_1 + \gamma_2$ is even}.
\end{equation}
The flip operator on $\Kmm$ is an involution, i.e. $\vect{\gamma}^{**} = \vect{\gamma}$. Further, if $\vect{\gamma} \in \Gamsquare$, then $\vect{\gamma}^* \in \Kmm \setminus\Gamsquare$ and $\gamma_1^* + \gamma_2^*$ is also an even number. 

Now, for an arbitrary subset $\Omega$ of $\Gamsquare$, we define the index set
\begin{equation} \label{eq:201807211209}
\GamD = \{ \Gamsquare \setminus \Omega\} \ \cup \ \{\vect{\gamma} \in \Kmm \ | \ \vect{\gamma}^* \in \Omega \}.
\end{equation}
By the considerations above, we have $\# \GamD = \# \Gamsquare = (2 m_1 +1 ) m_2$ and Theorem \ref{1507091911} combined with the glide-reflection 
symmetry \eqref{eq:glidereflection} of the basis functions $\polbas$ implies that also $\GamD$ is a spectral index set for $\Imm$. 

On the other hand, every spectral index set $\Gam \subset \Kmm$ given in Definition \ref{def:spectralindex}, contains $\# \Gam = (2 m_1 +1 ) m_2$ elements and the glide-reflection symmetry \eqref{eq:glidereflection} implies for $\vect{\gamma} \in \Gam$
that $\vect{\gamma}^* \in \Kmm \setminus \Gam$ and $\gamma_1^* + \gamma_2^*$ is even. 
By setting $\Omega = \Gam \cap \Gamsquare$, the set $\Gam$ is therefore identical 
to the spectral index set $\GamD$ given in \eqref{eq:201807211209}. We summarize these findings:

\begin{corollary}\label{1507091911BB} 
Every spectral index set $\Gam$ in Definition \ref{def:spectralindex} can 
be written in the form \eqref{eq:201807211209}, i.e., $\Gam = \GamD$ with $\Omega \subseteq \Gamsquare$.  
The basis functions $\polbas$, $\vect{\gamma} \in \Gam$, are normed by 
\begin{equation}\label{1508221825BB}
 \|\polbas\|_{\mathrm{w}}^2 =
 \left\{ \begin{array}{rl} 1,\; & \text{if}\quad \vect{\gamma} \in \Gam, \ \gamma_1 \in \{0,2m_1\},\\
 \frac12,\; & \text{if}\quad \vect{\gamma} \in \Gam, \ \gamma_1 \notin \{0,2m_1\}.
   \end{array} \right.
\end{equation}
\end{corollary}

\begin{example} The choice $\Omega = \emptyset$ gives exactly the rectangular spectral index set $\Gamsquare$. 
The choice $\Omega = \{ \vect{\gamma} \in \Gamsquare \ | \ \gamma_1/m_1 + \gamma_2/m_2 > 1, \ \gamma_1/m_1 - \gamma_2/m_2 > 1 \}$ gives 
a triangular spectral index set $\Gamtriangle$ of the form
\begin{equation}\label{1508222042triangle}
\Gamtriangle = \left\{\,\vect{\gamma}\in \Kmm \left| \ \begin{array}{l} \gamma_1/m_1 + |\gamma_2|/m_2 < 1 \\ \gamma_1 + \gamma_2 \ \text{is even} \end{array} \right. \right\} \cup 
\left\{\,\vect{\gamma}\in \Gamsquare \left| \ \gamma_1/m_1 + |\gamma_2|/m_2 = 1 \right. \right\}.
\end{equation}
In the forthcoming applications we will mostly use the rectangular spectral index set $\Gamsquare$ or the triangular set $\Gamtriangle$. For the frequency parameter $\vect{m} = (5,3)$, these two spectral
index sets are shown in Figure \ref{fig:LS-3}.
\end{example}

\paragraph{\bf 5.4. Real basis systems} For computational issues it is convenient to have also a real orthogonal basis for the space $\funpol(\Imm)$ at disposition. For this, we define the subset
\begin{equation} \label{eq:201807211209b}
\Gamcomp = \left\{ \vect{\gamma} \in \Gam \ | \ (\gamma_1,-\gamma_2) \notin \Gam \right\}
\end{equation}
and the real valued discrete functions
\begin{equation}\label{1702291124}
\polbasreal = \left\{  \begin{array}{lll}
\cos(\gamma_{1} i_1 \pi/(2m_{1})) \cos(\gamma_{2} i_2 \pi/(2m_{2})),
                                        & \text{if $\vect{\gamma} \in \Gam \setminus \Gamcomp$, $\gamma_2 \geq 0$}, \; & (i)\\
\cos(\gamma_{1} i_1 \pi/(2m_{1})) \sin(\gamma_{2} i_2 \pi/(2m_{2})),
                                        & \text{if $\vect{\gamma} \in \Gam \setminus \Gamcomp$, $\gamma_2 < 0$},  & (ii)\\
\cos(\gamma_{1} i_1 \pi/(2m_{1})) \cos(\gamma_{2} i_2 \pi/(2m_{2})),
                                        & \text{if $\vect{\gamma} \in \Gamcomp$, $\gamma_1 \leq m_1$},  & (iii)\\
\cos(\gamma_{1} i_1 \pi/(2m_{1})) \sin(\gamma_{2} i_2 \pi/(2m_{2})),
                                        & \text{if $\vect{\gamma} \in \Gamcomp$, $\gamma_1 > m_1$}.  & (iv)
\end{array} \right.
\end{equation}

\begin{theorem}\label{1507091912}
Let $\Gam$ be a spectral index set according to Definition \ref{def:spectralindex}.
Then, the functions $\polbasreal$, $\vect{\gamma}\in \Gam$, form a real
orthogonal basis of the inner product space $(\funpol(\Imm),\langle\,\cdot,\cdot\,\rangle_{\mathrm{w}})$.
The norms of the basis functions $\polbasreal$ are given as
\begin{equation}\label{1508221826}
 \|\polbasreal\|_{\mathrm{w}}^2 =
 \left\{ \begin{array}{rl}     1\; & \text{if}\quad \vect{\gamma} \in \{(0,0), (2 m_{1},0), (0,2m_2) \}, \\[2mm]
                             \frac12\; & \text{if} \quad \begin{array}{l} \vect{\gamma} \in \Gam \setminus \{(0,0), (2 m_{1},0), (0,2m_2) \} \ \text{and} \\                               
                             \text{$\gamma_1 = 0$ or $\gamma_1 = 2 m_1$ or $\gamma_2 = 0$ or $\gamma_2 = 2m_2$}, \end{array} \\[4mm]
                             \frac12\; & \text{if}\quad \vect{\gamma} \in \{(m_1,m_2), (m_1,-m_2) \}, \\[2mm]
                             \frac14\; & \text{for all other $\vect{\gamma} \in \Gam$}.
\end{array} \right.
\end{equation}
\end{theorem}

\begin{example}
\begin{enumerate}
\item[(i)] For the spectral index set $\Gamsquare$, we have $\Gamcomp_{\square} = \{ \vect{\gamma} \in \Gamsquare \ | \ \gamma_2 = m_2\}$. 
\item[(ii)] For the spectral index set $\Gamtriangle$, the set $\Gamcomp_{\triangle}$ is empty if $m_1 + m_2$ is odd. 
If $m_1+m_2$ is even, then $\Gamcomp_{\triangle} = \{(m_1,m_2)\}$. 
\end{enumerate}
\end{example}

\section{Spectral interpolation on the rhodonea nodes} \label{sec:interpolation}

 \begin{figure}[htb]
 	\centering
 	\subfigure[The basis $X_{\mathcal{R},\vect{\gamma}}$, $\vect{\gamma} \in \Gamtriangle$, $\vect{m} = (2,3)$.
 	]{\includegraphics[scale=0.2]{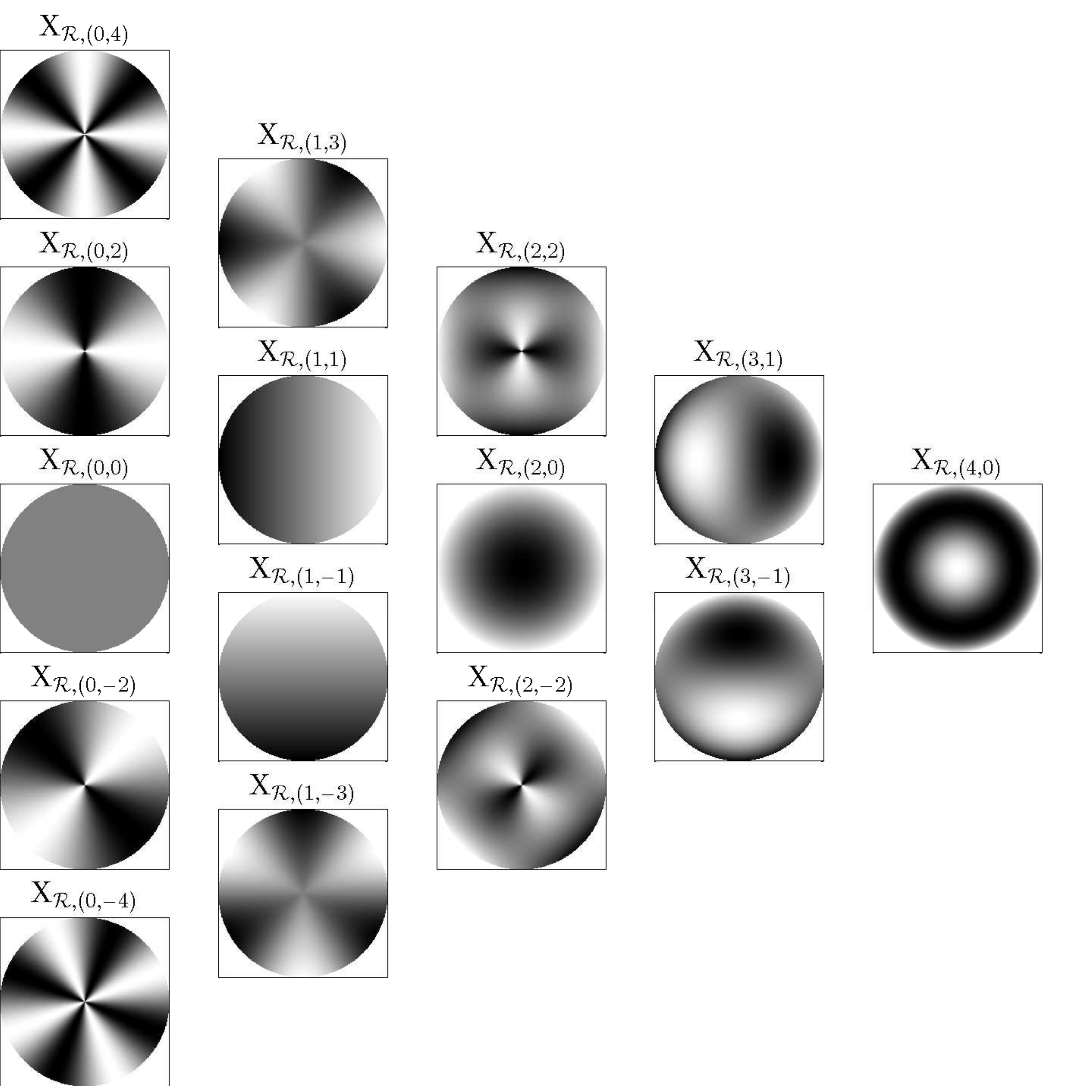}}
 	\hfill	
 	\subfigure[The basis $X_{\mathcal{R},\vect{\gamma}}$, $\vect{\gamma} \in \Gamsquare$, $\vect{m} = (2,3)$.
 	]{\includegraphics[scale=0.2]{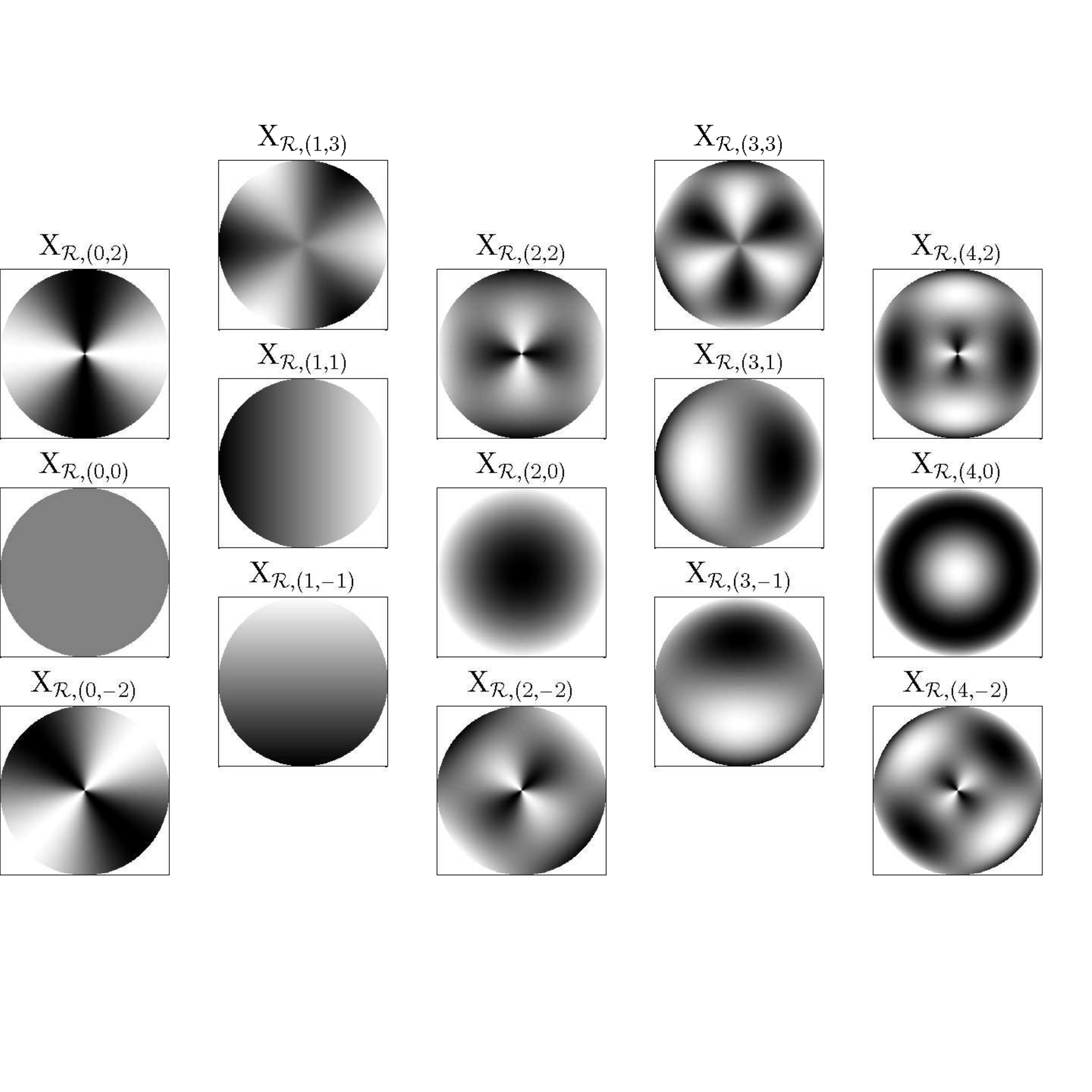}}
   	\caption{Illustration of the basis functions $\polbascont$ in the interpolation spaces $\Pimrealtriangle$ and $\Pimrealsquare$.
   	} \label{fig:LS-3B}
 \end{figure}

\paragraph{\bf 6.1. Formulation of the interpolation problem}
We formulate now the spectral interpolation problem on the disk $\Dd$ based on the rhodonea nodes $\LSm$ as interpolation nodes. As coordinate system, we will use polar coordinates $(r,\theta)$ in the domain $[0,1] \times (-\pi,\pi]$. 
According to the definitions in \eqref{eq:09172} and \eqref{eq:0917234},
the rhodonea nodes in polar coordinates are given as 
$(r^{(m_1)}_{i_1}, \theta^{(m_2)}_{i_2})$, $\vect{i} \in \Imm$. 

We generate the interpolation spaces using the \emph{Chebyshev-Fourier basis functions} $\polbascont$. 
For $\vect{\gamma} \in \Nn_0 \times \Zz$ and $(r,\theta) \in [0,1] \times (-\pi,\pi]$, these basis functions $\polbascont(r,\theta)$ are given by
\begin{equation}\label{A1508291123}
\polbascont(r,\theta) = T_{\gamma_1}(r) \mathrm{e}^{\imath \gamma_{2} \theta },
\end{equation}
where $T_{\gamma_1}(r) = \cos (\gamma_1 \arccos r)$ is the Chebyshev polynomial of first kind of degree $\gamma_1 \in \Nn_0$.
The space of all linear combinations of the functions $\polbascont$, $\vect{\gamma} \in \Nn_0 \times \Zz$,
is denoted by $\Pi$. 

Our aim is to solve the following interpolation problem: for given data values $f \in \funpol(\Imm)$ we want to obtain a spectral interpolant $ P^{(\vect{m})}_f \in \Pi$ such that
\begin{equation}\label{1508220011}
 P^{(\vect{m})}_f (r^{(m_1)}_{i_1}, \theta^{(m_2)}_{i_2}) = f({\vect{i}}) \quad \text{for all} \ \vect{i} \in \Imm.
\end{equation}

\paragraph{\bf 6.2. Unisolvence of spectral interpolation}
To obtain uniqueness in \eqref{1508220011}, we have to specify a proper subspace of $\Pi$. For this, 
we will use the key relation
\begin{equation}\label{1508201411}
\polbascont(r^{(m_1)}_{i_1}, \theta^{(m_2)}_{i_2}) = \polbas(\vect{i}), \qquad \vect{\gamma}\in\Nn_0 \times \Zz, \quad \vect{i}\in \Imm,
\end{equation}
between the Chebyshev-Fourier basis $\polbascont$ and the discrete basis $\polbas$ for the space $\funpol(\Imm)$.
From the previous section we know that $\polbas$, $\vect{\gamma} \in \Gam$, is an orthogonal 
basis of $\funpol(\Imm)$ if $\Gam$ is a spectral index set for $\Imm$. This turns spectral 
index sets $\Gam$ also to ideal index sets for the construction of the interpolation spaces. We define:
\[\Pim = \mathrm{span} \left\{\, \polbascont \,\left|\, \vect{\gamma} \in \Gam \right. \right\}.\]
By Definition \ref{def:spectralindex} of the spectral index set $\Gam$ the sum $\gamma_1 + \gamma_2$ is even. This parity condition ensures that the functions $P \in \Pim$ can be extended naturally onto $[-1,1] \times (-\pi,\pi]$ such that the 
continuous glide-reflection symmetry 
$P(-r,\theta) = P(r,\theta+\pi)$
is satisfied. For this reason, the basis $X_{\vect{\gamma}}$, $\vect{\gamma} \in \Gam$, is
also referred to as \emph{parity-modified Chebyshev-Fourier basis} of the space $\Pim$. 
Two such basis systems are illustrated in Figure \ref{fig:LS-3B}.

Data values $f \in \funpol(\Imm)$ obtained by sampling a continuous function on the disk are constant at all coordinates $\vect{i}$ with $i_1 = m_1$ representing 
the center of the unit disk. To take this fact into account, we additionally define the subspaces 
\[\funpols(\Imm) = \left\{f \in \mathcal{L}(\Imm): \, f(\vect{i}) = f_{\mathrm{C}} \ \text{is constant at $i_1 = m_1$} \right\}\]
and
\begin{equation}
\Pims = \left\{P \in \Pim \ | \ P(\theta^{(m_1)}_{i_1}, \vph^{(m_2)}_{i_2})) \equiv P(\theta^{(m_1)}_{j_1}, \vph^{(m_2)}_{j_2})) 
\quad \text{if} \; i_1 = j_1 = m_1 \  \right\}. \label{eq:186181745}
\end{equation}
We have $\funpols(\Imm) \subset \funpol(\Imm)$, $\Pims \subset \Pim$ and $\dim \funpols(\Imm) = \dim \Pims = \# \LSm$. 
The space $\funpols(\Imm)$ can naturally be used to describe all given data functions on the rhodonea
nodes $\LSm$. On the other side, the space $\Pims \subset \Pim$ contains exactly all $P \in \Pim$ such that the discrete data set $p(\vect{i}) = P(r^{(m_1)}_{i_1}, \theta^{(m_2)}_{i_2}))$, 
$\vect{i} \in \Imm$, is contained in $\funpols(\Imm)$. Although $P \in \Pims$ satisfies this discrete consistency condition at the center of $\Dd$, the
function $P \in \Pims$ is in general not constant on the entire line $r = 0$ describing the center. We will show in Section \ref{sec:convergence}
that for the particular interpolation space $\Pimssquare$ based on the index set $\Gamsquare$ we can guarantee the continuity of $P$ at the center.

\begin{remark}
In the literature on spectral spectral methods a tensor-product grid in polar
coordinates is usually used in place of the rhodonea points to build up collocation schemes on the unit disk. The underlying 
interpolation spaces spanned by a parity-modified Chebyshev-Fourier basis in a rectangular spectral set are similar to the spaces
$\Pimsquare$, see \cite{BoydYu2011,Fornberg1995,Shen2011,Trefethen2000,TownsendWilberWright2017}. 

In our considerations, the spaces $\Pimsquare$ play a dominant role as well. Nevertheless, also $\Pimtriangle$ has some interesting resemblances to spaces in other works.
For $m \in \Nn$ odd and $\vect{m} = (m-1,m)$, the interpolation space $\Pimtriangle$ is spanned by all parity-modified Chebyshev-Fourier basis functions
$\polbascont$ with total degree $|\gamma_1| + |\gamma_2| \leq m-1$. Because of this,
the rhodonea nodes $\LS^{(m-1,m)}$ can be regarded as polar version of
the Padua points studied in \cite{BosDeMarchiVianelloXu2006,CaliariDeMarchiVianello2005}. 
If $\vect{m} = (m,m)$, the points $\LSm$ and 
the interpolation space $\Pimtriangle$ provide a setup that is very similar to the one provided by the Morrow-Patterson-Xu points
\cite{Xu1996}. For general $\vect{m} \in \Nn^2$, the results obtained for the nodes $\LSm$ can be regarded as a polar version of the theory on polynomial interpolation on Lissajous nodes \cite{DenckerErb2017a,DenckerErb2015a,Erb2015,ErbKaethnerAhlborgBuzug2015} and on 
spherical Lissajous nodes \cite{ErbSphere2017}. 
\end{remark}

Our main result on spectral interpolation on the disk reads as follows:
\begin{theorem} \label{201512131945}
Let $f\in \funpol(\Imm)$ and $\Gam$ be a spectral index set according to Definition
\ref{def:spectralindex}. Then, the interpolation problem \eqref{1508220011} has a unique solution
in the space $\Pim$ given by
\begin{equation} \label{201513121708}
P^{(\vect{m})}_{f}(r,\theta) = \sum_{\vect{i} \in \Imm} f({\vect{i}}) \, L^{(\vect{m})}_{\vect{i}}(r,\theta)
\end{equation}
with the Lagrange functions 
\begin{equation} \label{1508220009}
 L^{(\vect{m})}_{\vect{i}}(r,\theta) = \mathrm{w}^{(\vect{m})}_{\vect{i}} \sum_{\vect{\gamma} \in \Gam} \frac{\overline{\polbas(\vect{i})}}{\|\polbas\|_{\mathrm{w}}^2}
 \polbascont(r,\theta), \qquad \vect{i} \in \Imm,
\end{equation}
forming a basis of the vector space $\Pim$.
For $f\in \funpols(\Imm)$, the interpolant $P^{(\vect{m})}_{f} \in \Pims$ is of the form
\begin{equation} \label{201513121700}
P^{(\vect{m})}_{f}(r,\theta) = \underset{\vect{i} \in \Imm: i_1 \neq m_1}{\sum} f({\vect{i}}) \, L^{(\vect{m})}_{\vect{i}}(r,\theta) + f_{\mathrm{C}} \underset{\vect{i} \in \Imm: i_1 = m_1}{\sum} L^{(\vect{m})}_{\vect{i}}(r,\theta).
\end{equation}
\end{theorem}

\paragraph{\bf 6.3. Real valued interpolation spaces}
In order to establish a similar interpolation result for real vector spaces, we 
define for $\vect{\gamma} \in \Gam$ the real basis functions
\begin{equation*}
\polbascontreal(r,\theta) = \left\{ \begin{array}{ll} T_{\gamma_{1}} (r) \cos (\gamma_{2} \theta ) & \text{if $\vect{\gamma} \in \Gam \setminus \Gamcomp$, $\gamma_2 \geq 0$}, \\
                                 T_{\gamma_{1}} (r) \sin (\gamma_{2} \theta ) & \text{if $\vect{\gamma} \in \Gam \setminus \Gamcomp$, $\gamma_2 < 0$}, \\
                                 T_{\gamma_{1}} (r) \cos (\gamma_{2} \theta ) & \text{if $\vect{\gamma} \in \Gamcomp$, $\gamma_1 \leq m_1$}, \\
                                 T_{\gamma_{1}} (r) \sin (\gamma_{2} \theta ) & \text{if $\vect{\gamma} \in \Gamcomp$, $\gamma_1 > m_1$}. \\
\end{array} \right.
\end{equation*}

Evaluating the functions $\polbascontreal$ at the polar nodes $(r^{(m_1)}_{i_1}, \theta^{(m_2)}_{i_2})$, we obtain precisely the discrete basis functions of the space $\funpol(\Imm)$ given in \eqref{1702291124}, i.e.
\begin{equation}\label{1508201413}
\polbascontreal(r^{(m_1)}_{i_1}, \theta^{(m_2)}_{i_2}) = \polbasreal(\vect{i}), \qquad \text{$\vect{\gamma}\in\Gam$, $\vect{i}\in \Imm$}.
\end{equation}
In the same way as for the basis function $\polbascont$ and the space $\Pim$, we can now introduce the real valued interpolation space
\[\Pimreal = \mathrm{span} \left\{\, \polbascontreal \,\left|\, \vect{\gamma} \in \Gam \right. \right\}\]
and obtain in analogy to Theorem \ref{201512131945} the following result:

\begin{theorem} \label{201512131946}
Let $f\in \funpol(\Imm)$ and $\Gam$ be a spectral index set according to Definition
\ref{def:spectralindex}. Then, the interpolation problem \eqref{1508220011} has a unique solution
in the real space $\Pimreal$ given by 
\begin{equation} \label{201513121708B}
P^{(\vect{m})}_{\mathcal{R},f}(r,\theta) = \sum_{\vect{i} \in \Imm} f({\vect{i}}) \, L^{(\vect{m})}_{\mathcal{R},\vect{i}}(r,\theta).
\end{equation}
with the Lagrange functions 
\begin{equation} \label{1508220009B}
 L^{(\vect{m})}_{\mathcal{R},\vect{i}}(r,\theta) = \mathrm{w}^{(\vect{m})}_{\vect{i}} 
 \sum_{\vect{\gamma} \in \Gam} \frac{\polbasreal(\vect{i})}{\|\polbasreal\|_{\mathrm{w}}^2}
 \polbascontreal(r,\theta), \qquad \vect{i} \in \Imm,
\end{equation}
forming a basis of the vector space $\Pimreal$. If $f\in \funpols(\Imm)$, the solution of the interpolation problem \eqref{1508220011} can be written as
\begin{equation} \label{201513121700B}
P^{(\vect{m})}_{\mathcal{R},f}(r,\theta) = \underset{\vect{i} \in \Imm: i_1 \neq m_1}{\sum} f({\vect{i}}) \, L^{(\vect{m})}_{\mathcal{R},\vect{i}}(r,\theta) + f_{\mathrm{C}} \underset{\vect{i} \in \Imm: i_1 = m_1}{\sum} L^{(\vect{m})}_{\mathcal{R},\vect{i}}(r,\theta).
\end{equation}
\end{theorem}

\begin{remark}
In the discrete setting the basis systems $\polbas$ and $\polbasreal$, $\vect{\gamma} \in \Gam$ 
generate the same vector space $\funpol(\Imm)$. In the continuous setup this is no longer true and 
we generally have $\Pim \neq \Pimreal$. An example in which $\Pim$ coincides with $\Pimreal$ can be obtained for 
the triangular spectral index set $\Gamtriangle$. We get $\Pimtriangle = \Pimrealtriangle$ in the case that $m_1 + m_2$ is odd. This follows from the fact that in this case $\Gamcomp_{\triangle}$ is empty.
\end{remark}

\section{Efficient implementation of the interpolation algorithm} \label{sec:implementation}

\paragraph{\bf \ref{sec:implementation}.1. Calculation of the expansion coefficients}

An efficient way to calculate the interpolation polynomial $P^{(\vect{m})}_{f} \in \Pim$ from given data values $f \in \Imm$ is based on the expansion
\begin{equation} \label{20170303146} P^{(\vect{m})}_{f}(r,\theta) = \sum_{\vect{\gamma} \in \Gam} c_{\vect{\gamma}}(f) \polbascont(r,\theta). \end{equation}
Using this series expansion, $P^{(\vect{m})}_{f}$ can be evaluated once the coefficients $c_{\vect{\gamma}}(f)$ are calculated. Both steps, the calculation of the coefficients and the evaluation of the sum \eqref{20170303146} can be implemented by applying a discrete Fourier transform.
Theorem \ref{201512131945} and the definition $\eqref{1508220009}$ of the Lagrange basis functions 
provide us with the representation
\[P^{(\vect{m})}_{f}(r,\theta) = \sum_{\vect{\gamma} \in \Gam} \frac{1}{\|\polbas\|_{\mathrm{w}}^2} \left(\sum_{\vect{i} \in \Imm}
\mathrm{w}^{(\vect{m})}_{\vect{i}} f({\vect{i}}) \overline{\polbas(\vect{i})} \right) \polbascont(r,\theta).\]
Since the set $\{ \polbascont \ | \ \vect{\gamma} \in \Gam\}$ is a basis for $\Pim$, we have the identity
\[c_{\vect{\gamma}}(f) = \frac{1}{\|\polbas\|_{\mathrm{w}}^2} \left(\sum_{\vect{i} \in \Imm}
\mathrm{w}^{(\vect{m})}_{\vect{i}} f({\vect{i}}) \overline{\polbas(\vect{i})} \right) = \frac{1}{\|\polbas\|_{\mathrm{w}}^2}\,\ds \langle\;\! f,\polbas \rangle_{\mathrm{w}}.\]
This formula enables us to calculate the expansion coefficients $c_{\vect{\gamma}}(f)$ by a double Fourier transform on the finite abelian group $( \Zz / 4 m_1 \Zz) \times ( \Zz / 4 m_2 \Zz)$ identified with
\[ \Jmm = \{ \vect{i} \in \Zz^2 \ | \ -2m_1 < i_1 \leq 2 m_1, \ -2 m_2 < i_2 \leq 2 m_2 \}. \]
We consider $\Kmm$ as a subset of $\Jmm$ and the flip operator $\phantom{}^{*}$ introduced in \eqref{eq:201807201118} on $\Kmm$. Further, we introduce
a second reflection operator on $\Jmm$ by setting $\vect{i}^\dag = (- i_1 \mod 4m_1 , i_2 )$ for $\vect{i} \in \Jmm$. The two mappings allow us to extend $f \in \Imm$ symmetrically to $\Jmm$. We set
\begin{equation} \label{eq:g} g(\vect{i}) = \frac{1}{8 m_1m_2}  \left\{ \begin{array}{rl} f(\vect{i}), \quad  & \text{if}\; \vect{i}\in\Imm, \rule[-0.65em]{0pt}{1em}\\
f(\vect{i}^*), \quad  & \text{if}\; \vect{i}\in\Kmm, \; \vect{i}^*\in\Imm,
\rule[-0.65em]{0pt}{1em}\\
f(\vect{i}^\dag), \quad  & \text{if}\; \vect{i}\in\Jmm \setminus \Kmm, \; \vect{i}^\dag\in\Imm,
\rule[-0.65em]{0pt}{1em}\\
f(\vect{i}^{\dag *} ), \quad  & \text{if}\; \vect{i}\in\Jmm \setminus \Kmm, \; \vect{i}^{\dag *}\in\Imm,
\rule[-0.65em]{0pt}{1em}\\
       0, \quad
  & \text{otherwise}. \end{array} \right. \end{equation}
The coefficients $c_{\vect{\gamma}}(f)$, $\vect{\gamma} \in \Gam$ can now be obtained directly from the Fourier transform 
\[\hat{g}(\vect{\gamma}) =\sum_{\vect{i} \in \Jmm} g(\vect{i}) \mathrm{e}^{-\imath \gamma_{1} i_1 \pi/m_{1} } \mathrm{e}^{-\imath \gamma_{2} i_2 \pi/m_{2} } \quad \text{of the function $g$ on $\Jmm$.} \]
The relation between $\hat{g}(\vect{\gamma})$ and the coefficients $c_{\vect{\gamma}}(f)$ is herein given by
\begin{align} c_{\vect{\gamma}}(f)&=\ts \frac{1}{\|\polbas\|_{\mathrm{w}}^2}\,\ds \sum_{\vect{i} \in \Imm} \mathrm{w}^{\!(\vect{m})}_{\vect{i}} f(\vect{i}) \overline{\polbas(\vect{i})} 
 = \ts \frac{1}{\|\polbas\|_{\mathrm{w}}^2}\,\ds \sum_{\vect{i} \in \Jmm} g(\vect{i}) \mathrm{e}^{-\imath \gamma_{1} i_1 \pi/m_{1} } \mathrm{e}^{-\imath \gamma_{2} i_2 \pi/m_{2} } \notag \\
  &= \left\{ \begin{array}{rl} 2 \hat{g}(\vect{\gamma}),  & \text{if}\, \vect{\gamma}\in\Gam, \, \gamma_1 \notin \{0,2m_1\},
\rule[-0.65em]{0pt}{1em}\\ \hat{g}(\vect{\gamma}),  & \text{if}\,  \vect{\gamma}\in\Gam, \, \gamma_1 \in \{0,2m_1\}. \end{array} \right. \label{20170304}
\end{align}
The entire calculation of the coefficients is summarized in Algorithm \ref{algorithm1}. The main
computational step in Algorithm \ref{algorithm1} consists in the calculation of the Fourier transform 
$\hat{g}$. By using a fast Fourier algorithm this step can be executed in $\Ord(m_1m_2 \log (m_1m_2))$ arithmetic operations. The values of $\|\polbas\|_{\mathrm{w}}^2$ 
used in \eqref{20170304} are known from \eqref{1508221825}.

\begin{remark}
The symmetry of the function $g$ can be seen as a combination of a reflection and a glide reflection symmetry. The reflection
symmetry corresponds to the invariance of $g$ under the reflection operator $\vect{i}^{\dag}$ on $\Jmm$, while the glide reflection 
symmetry is described by the operator $\vect{i}^{*}$ on the subset $\Kmm$. In \cite{TownsendWilberWright2016,TownsendWilberWright2017}, 
this glide reflection symmetry is referred to as block-mirror centrosymmetric (BMC) structure. 
\end{remark}

\renewcommand{\algorithmcfname}{Alg.}
\setlength{\algomargin}{1mm}
\begin{table}
\begin{multicols}{2}
\begin{algorithm}[H] \label{algorithm1}

\vspace{2mm}

\KwIn{$f \in \Imm$}

\vspace{1mm}

\textbf{Calculate} $g \in \Jmm$ from $f$ by \eqref{eq:g}  

\vspace{1mm}

\textbf{Calculate} Fourier transform $\hat{g}$ of $g$ on the group $\Jmm$

\vspace{1mm}

For $\vect{\gamma}\in\Gam$, \textbf{set} \[c_{\vect{\gamma}}(f)
  = \left\{ \begin{array}{rl} 2 \hat{g}(\vect{\gamma}),  & \text{if} \, \gamma_1 \notin \{0,2m_1\},
\rule[-0.65em]{0pt}{1em}\\ \hat{g}(\vect{\gamma}),  & \text{if}\, \gamma_1 \in \{0,2m_1\}. \end{array} \right. 
\]

\vspace{-3mm}

\caption{Calculation of coefficients}
\end{algorithm}

\columnbreak

\begin{algorithm}[H] \label{algorithm2}

\vspace{2mm}

\KwIn{$c_{\vect{\gamma}}(f)$ on $\Gam$}

\vspace{1mm}

\textbf{Calculate} $h \in \Jmm$ from $c_{\vect{\gamma}}(f)$ by  \eqref{eq:h}  

\vspace{1mm}

\textbf{Calculate} adjoint Fourier transform $\check{h}$ of $h$ on the (dual) group $\Jmm$

\vspace{1mm}

For $\vect{i}\in\Imm$, \textbf{set} \[ f(\vect{i}) = \check{h}(\vect{i}).\]

\vspace{4.5mm}

\caption{Inverse transform}
\end{algorithm}

\end{multicols}
\caption{Algorithms for the calculation of the expansion coefficients and the inverse transform.}
\end{table}

\paragraph{\bf \ref{sec:implementation}.2. The inverse transform}
From a known set of coefficients $c_{\vect{\gamma}}(f)$, $\vect{\gamma} \in \Gam$, we can reversely reconstruct the function values $f \in \mathcal{L}(\Imm)$. This inverse transform is also 
determined by a discrete Fourier transform. Combining the interpolation condition \eqref{1508220011} 
with the expansion \eqref{20170303146}, we get
\begin{equation*}
 f({\vect{i}}) = P^{(\vect{m})}_f (r^{(m_1)}_{i_1}, \theta^{(m_2)}_{i_2}) = 
 \sum_{\vect{\gamma} \in \Gam} c_{\vect{\gamma}}(f) \polbas(\vect{i}) \quad \text{for all}\quad \vect{i} \in \Imm.
\end{equation*}
As in the previous section, we extend the coefficients $c_{\vect{\gamma}}(f)$ first symmetrically to the (dual) group $\Jmm$ and define the function $h$ on $\Jmm$ as 
\begin{equation} \label{eq:h} h(\vect{\gamma}) =  \left\{ \begin{array}{rl} 2 c_{\vect{\gamma}}(f), \quad  & \text{if}\; \vect{\gamma}\in\Gam, \, \gamma_1 \notin \{0,2 m_1\},
\rule[-0.65em]{0pt}{1em}\\
2 c_{\vect{\gamma}}(f), \quad  & \text{if}\; (-\gamma_1,\gamma_2)\in\Gam, \, \gamma_1 \notin \{0,2 m_1\},
\rule[-0.65em]{0pt}{1em}\\ 
 c_{\vect{\gamma}}(f), \quad  & \text{if}\; \vect{\gamma}\in\Gam, \, \gamma_1 \in \{0,2m_1\},
\rule[-0.65em]{0pt}{1em}\\ 
       0, \quad
  & \text{otherwise}. \end{array} \right. \end{equation}
This definition together with the definition \eqref{A1508291531} of the discrete basis functions $\polbas$ yields
\begin{equation*}
 f({\vect{i}}) = \sum_{\vect{\gamma} \in \Jmm} h(\vect{\gamma}) \mathrm{e}^{\imath \gamma_{1} i_1 \pi/m_{1} } \mathrm{e}^{\imath \gamma_{2} i_2 \pi/m_{2} }  =  \check{h}(\vect{i})
 \quad \text{for}\; \vect{i} \in \Imm.
\end{equation*} 
Therefore, the function $f \in \mathcal{L}(\Imm)$ can be recovered from the coefficients $c_{\vect{\gamma}}(f)$ by computing the adjoint Fourier transform $\check{h}$ of 
$h$ on $\Jmm$. The single steps of the calculation are summarized in Algorithm \ref{algorithm2}. 
As for Algorithm \ref{algorithm1}, the entire inverse transform can be computed 
in $\Ord(m_1m_2 \log (m_1m_2))$ arithmetic operations.

\paragraph{\bf \ref{sec:implementation}.3. Calculation of real expansion coefficients}
When working with the real basis $\polbascontreal$, $\vect{\gamma} \in \Gam$, the expansion coefficients $c_{\mathcal{R},\vect{\gamma}}(f)$ of the interpolant $P^{(\vect{m})}_{\mathcal{R},f}$
can be computed in a similar way. Using the formula
\[c_{\mathcal{R},\vect{\gamma}}(f)=\ts \frac{1}{\|\polbasreal\|_{\mathrm{w}}^2}\,\ds \langle\;\! f,\polbasreal \rangle_{\mathrm{w}},\]
the coefficients can be rewritten as
\begin{equation*}
c_{\mathcal{R},\vect{\gamma}}(f) = \frac{1}{\|\polbasreal\|_{\mathrm{w}}^2} \left\{ \! \begin{array}{ll}
\Re \hat{g}(\vect{\gamma})
                                        & \text{if $\vect{\gamma} \in \Gam \setminus \Gamcomp$, $\gamma_2 \geq 0$}, \\
- \Im \hat{g}(\vect{\gamma})
                                        & \text{if $\vect{\gamma} \in \Gam \setminus \Gamcomp$, $\gamma_2 < 0$}, \\
\Re \hat{g}(\vect{\gamma})
                                        & \text{if $\vect{\gamma} \in \Gamcomp$, $\gamma_1 \leq m_1/2$},\\
- \Im \hat{g}(\vect{\gamma})
                                        & \text{if $\vect{\gamma} \in \Gamcomp$, $\gamma_1 > m_1/2$}.
\end{array} \right.
\end{equation*}
This formula can be derived as in \eqref{20170304} by using the real basis $\polbasreal$ in \eqref{1702291124} instead of the complex-valued functions $\polbas$.
The values $\|\polbasreal\|_{\mathrm{w}}^2$ are explicitly known from \eqref{1508221826} and the calculation of the Fourier transform $\hat{g}(\vect{\gamma})$ is the same as in Section \ref{sec:implementation}.1.

\paragraph{\bf \ref{sec:implementation}.4. Averaged expansion coefficients}
A further option to alter the structure of the interpolation spaces is to use, for some of the indices $\vect{\gamma} \in \Gam$ ,
the averaged basis functions $X_{\mathcal{A},\vect{\gamma}} = \lambda X_{\vect{\gamma}} + (1 - \lambda) X_{\vect{\gamma^*}}$ with $\lambda \in \Rr$ instead of the standard basis functions $X_{\vect{\gamma}}$.
Since the discrete basis functions $\polbas$, $\vect{\gamma} \in \Gam$, are invariant under the flip operator $\phantom{}^{*}$, we have 
\[X_{\vect{\gamma}}(r_{i_1}^{(m_1)},\theta_{i_2}^{(m_2)}) = \polbas(\vect{i}) = X_{\mathcal{A},\vect{\gamma}}(r_{i_1}^{(m_1)},\theta_{i_2}^{(m_2)}).\]
Therefore, while both basis systems define different interpolation spaces, they both lead to the same interpolation problem on the rhodonea nodes and the expansion coefficients are identical. 
Interpolation spaces with such an averaging for some of the boundary elements
of the spectral index set were originally studied in the bivariate setting for the Morrow-Patterson-Xu points in \cite{Harris2010,Xu1996}. For multivariate interpolation on Lissajous-Chebyshev nodes, this
averaging process is studied in more detail in \cite{DenckerErb2017a}.

\begin{example} An alternative to the spectral index set $\Gamsquare$ in the implementation of the interpolation scheme is to
use the more symmetric set (compare the definition \eqref{1508222042base} of $\Gamsquare$)
\[ \Gamsupsquare = \left\{\,\vect{\gamma}\in \Kmm \left| \ - m_2 \leq \gamma_{2} \leq m_{2}, \ 
\gamma_1 + \gamma_2 \ \text{is even}
\right.
\right\}\]
and to average the basis functions for the coefficients $\vect{\gamma}$, $|\gamma_2| = m_2$ at the
upper and lower boundary of $\Gamsupsquare$ with $\lambda = 1/2$. In this way we get an interpolation function of the form
\[P_{\mathcal{A},f}^{(\vect{m})}(r,\theta) = \sum_{\vect{\gamma} \in \Gamsupsquare} c_{\mathcal{A},\vect{\gamma}}(f) X_{\vect{\gamma}},\]
in which the coefficients are given as
\[ c_{\mathcal{A},\vect{\gamma}}(f) = \left\{ \begin{array}{ll} c_{\vect{\gamma}}(f)/2 & \text{if} \ \vect{\gamma} \in \Gamsupsquare, \ |\gamma_2| = m_2, \\
c_{\vect{\gamma}}(f) & \text{if} \ \vect{\gamma} \in \Gamsupsquare, \ |\gamma_2| \neq m_2.  \end{array}\right.\]
The coefficients $c_{\vect{\gamma}}(f)$ itself are calculated as in \eqref{20170304}. For the real basis functions $X_{\mathcal{R},\vect{\gamma}}$ similar averaging strategies are of course also possible.   
\end{example}

\section{Properties of the spectral interpolation scheme} \label{sec:convergence}

Goal of this section is to provide a convergence analysis of the presented interpolation scheme 
and to answer questions typically considered in approximation theory and
numerical analysis. This includes quadrature possibilities, the behavior of the
interpolating functions at the center of the unit disk, as well as the numerical condition
and the convergence of the interpolation scheme if the number of nodes gets large. In general, these properties depend on the geometric form of the spectral index set $\Gam$. To obtain more concrete results, we will restrict our studies mainly to the two particular spectral sets $\Gamsquare$ and $\Gamtriangle$.  

The interpolating functions $P^{(\vect{m})}_{f}$ considered in this section are determined by data values $f$ that are obtained from the samples of a continuous function on the disk. For a continuous function $\ff(r,\theta)$ on $\Dd$ in polar coordinates, we have the relation
\begin{equation} \label{eq:186181822}
f(\vect{i}) = \ff(r^{(m_1)}_{i_1}, \theta^{(m_2)}_{i_2}) \quad \text{for}\quad \vect{i} \in \Imm.
\end{equation}
In particular, $f \in \funpols(\Imm)$ and Theorem \ref{201512131945} ensures that we obtain a unique interpolant $P^{(\vect{m})}_{f}$ in $\Pims \subset \Pim$ that interpolates the function $\ff$ at 
the rhodonea nodes $\LSm$.

\paragraph{\bf \ref{sec:convergence}.1. Continuity at the center} 
When using polar coordinates $(r,\theta) \in [0,1] \times [-\pi,\pi]$ to describe a continuous function $\ff$ on $\Dd$ we have to add the usual topological identifications for the polar 
coordinates. We can describe the space of all continuous functions on $\Dd$ by 
\[ C(\Dd) = \left\{\ff \in C([0,1]\times[-\pi,\pi]) \ \left| \ \begin{array}{lll} \mathrm{(i)} & \ff(r,-\pi) = \ff(r,\pi), & 0 \leq r \leq 1, \\
                                                                       \mathrm{(ii)} & \ff(0,\theta_1) = \ff(0,\theta_2),   & -\pi \leq \theta_1,\theta_2 \leq \pi. \end{array}
                                                                       \right. \right\} \]
Not all basis functions $\polbascont$, $\vect{\gamma} \in \Gam$, are contained in $C(\Dd)$. While $\polbascont \in C([0,1]\times[-\pi,\pi])$
and the periodicity $\mathrm{(i)}$ are always satisfied, the continuity $\mathrm{(ii)}$ at the center holds only true if $\gamma_1$ and $\gamma_2$ are odd. 
Therefore, also for the interpolant $P^{(\vect{m})}_{f}$ we can in general not expect that the continuity $\mathrm{(ii)}$ is satisfied. However, if 
the rectangular spectral index set $\Gamsquare$ is used for the interpolation space, we can guarantee that $P^{(\vect{m})}_{f}$ is continuous also at the center. 

\begin{theorem} 
\label{th:201807231722}
Let $\ff \in C(\Dd)$ and $P^{(\vect{m})}_{f}$ be the unique interpolant in the space $\Pimsquare$ according to Theorem \ref{201512131945}. Then,
 $P^{(\vect{m})}_{f} \in C(\Dd)$. 
\end{theorem}

\begin{remark} \label{rem:AB1}
Note that Theorem \ref{th:201807231722} does in general not hold true if $\Gamsquare$ is replaced by a different spectral index set $\Gam$. In general, the trigonometric polynomial $P^{(\vect{m})}_{f}(0,\theta)$ (see also \eqref{eq:1807221941} in the proof of Theorem \ref{th:201807231722}) is contained in a space of dimension larger than $m_2$ and the $m_2$ given boundary conditions can not guarantee that $P^{(\vect{m})}_{f}(0,\theta)$ is constant. Also for real interpolation spaces $\Pimrealsquare$ a careful view at the boundary conditions is necessary. For the interpolation
space $\Pimrealsquare$, the statement of Theorem \ref{th:201807231722} holds only true if $m_2$ is odd. 
\end{remark}

\begin{remark} Although Theorem \ref{th:201807231722} establishes that, at least in the case of the rectangular set $\Gamsquare$, the interpolant $P^{(\vect{m})}_{f}$ is continuous, we can not expect higher order smoothness of $P^{(\vect{m})}_{f}$ at the center of $\Dd$. In Section
\ref{sec:convergence}.3. we will see that these singularities of $P^{(\vect{m})}_{f}$ at the center have no influence on the global convergence of the interpolation scheme. If the function $\ff$ is sufficiently smooth, the derivatives of $P^{(\vect{m})}_{f}$ will approximately satisfy the continuity condition $\mathrm{(ii)}$ as soon as the node set $\LSm$ gets large. 
In \cite{Boyd1978}, such a property is referred to as natural boundary condition. For a lot of applications it is sufficient if such a natural boundary condition is satisfied. A deeper discussion about the behavior of spectral methods at coordinate singularities can be found in \cite{Boyd2000,HMS2002,Shen2011}. 
\end{remark}

\paragraph{\bf \ref{sec:convergence}.2. Numerical condition of the interpolation scheme}
The Lebesgue constant of the interpolation problem \eqref{1508220011} is defined as the operator norm
\[\Lambda^{(\vect{m})}_{\Omega} = \sup_{\|\ff\|_{\infty} \leq 1} \|P^{(\vect{m})}_{f}\|_{\infty}, \quad \text{with} \;\| \ff \|_{\infty} = \sup_{(r,\theta)} |\ff(r,\theta)|. \]
In numerical analysis, this constant is interpreted as the absolute condition number of the interpolation problem \eqref{1508220011}.
It gives an upper bound on how a small error in the function $\ff$ affects the corresponding interpolant $P^{(\vect{m})}_{f}$ in the uniform norm. 
Beside the distribution of nodes $\LSm$, the Lebesgue constant depends on the geometric structure of the spectral index set $\Gam$. For the rectangular spectral set
$\Gamsquare$, we can guarantee that this Lebesgue constant grows only slowly if $m_1$ and $m_2$ get large.
\begin{theorem} \label{thm:lebesgueconstant} For the interpolation space $\Pimsquare$, the Lebesgue constant $\Lambda^{(\vect{m})}$ is bounded by
\[ \Lambda^{(\vect{m})}_{\square} \leq C_{\square} \ln (m_1+1) \ln (m_2+1)\] 
with a constant $C_{\square}$ independent of $\vect{m}$. 
\end{theorem}

\begin{remark}
A similar logarithmic estimate holds true for the Lebesgue constant 
$\Gamma^{(\vect{m})}_{\triangle}$ when the spectral index set $\Gamsquare$ is replaced by the triangular set $\Gamtriangle$. In this case, techniques developed in \cite[Section 2]{DEKL2017} can be used 
to obtain the estimates of the resulting double integrals. For a similar setting on the unit sphere, the 
respective proof can be found in \cite{ErbSphere2017}. 
\end{remark}

\paragraph{\bf \ref{sec:convergence}.3. Convergence of the interpolation scheme}

Once an estimate for the Lebesgue constant is known, the convergence of $P^{(\vect{m})}_{f}$ towards
$\ff \in C(\Dd)$ can be established easily if the underlying function $\ff$ is smooth. Further, we obtain better rates of convergence the smoother the function $\ff$ is. This is a general principle 
for spectral methods in a variety of settings \cite{Boyd2000,Trefethen2000}. For multivariate polynomial interpolation in the hypercube $[-1,1]^{\mathsf{d}}$ similar results for Lissajous sampling nodes can be found in \cite{DEKL2017,Erb2015}. For spherical Lissajous nodes a respective derivation is given in \cite{ErbSphere2017}. Similar error estimates for a tensor product spectral collocation scheme on the sphere can also be found in \cite{Ganesh1998}.

We consider $P^{(\vect{m})}_{f}$ in the interpolation
space $\Pimsquare$. If $P^*$ denotes the best possible approximation of $\ff$ in 
$\Pimsquare$ the uniform error
$\|\ff - P^{(\vect{m})}_{f} \|_\infty$ can be bounded by
\begin{align*} \|\ff - P^{(\vect{m})}_{f} \|_\infty &\leq \|\ff - P^* \|_\infty + \| P^* - P^{(\vect{m})}_{f} \|_\infty \\ & \leq (\Lambda^{(\vect{m})}_{\square} + 1) \|f - P^* \|_\infty  
= (C_{\square}+1) \ln (m_1+1) \ln (m_2+1)  \|\ff - P^* \|_\infty.
\end{align*}
In the second estimate, we used the fact that the interpolation operator $\ff \to P^{(\vect{m})}_{f}$ reproduces $P^* \in \Pimsquare$ together with the estimate of the Lebesgue constant in Theorem \ref{thm:lebesgueconstant}. If $\ff$ is $s$ times continuously differentiable in $\Dd$, 
we can estimate the best error $\|\ff - P^* \|_\infty$ with help of a multivariate Jackson inequality for trigonometric functions (see \cite[Section 5.3]{Timan1960}). In this way, we get the error bound
\begin{equation} \label{eq:1806211617} \|\ff - P^{(\vect{m})}_{f}\|_\infty \leq C_{\ff,s} \ln(m_1+1) \ln(m_2+1) \left( \frac{1}{m_1^s} + \frac{1}{m_2^s}\right).\end{equation}
The constant $C_{\ff,s}$ is independent of $\vect{m}$. Thus, if $\ff$ is smooth, this estimate guarantees a fast uniform convergence of the interpolant towards $\ff$ if $m_1$ and $m_2$ get large. 

\paragraph{\bf \ref{sec:convergence}.4. Quadrature formula on the rhodonea nodes} \label{sec:cc}

In order to formulate a Clenshaw-Curtis quadrature rule for the rhodonea points $\LSm$, only 
the expansion \eqref{20170303146} and the explicit integration of 
the basis functions $\polbascont \in \Pim$ over the disk $\Dd$ are necessary.  
In polar coordinates the area element on $\Dd$ is given by $r \, \mathrm{d} r \mathrm{d} \theta$. The tensor product structure of the Chebyshev-Fourier basis functions $\polbascont$ then yields the formula
\[ \frac{1}{\pi} \int_0^{2\pi} \int_0^{1} \polbascont(r,\theta) r \, \mathrm{d} r \mathrm{d} \theta = \left\{ \begin{array}{ll}
 \frac{1}{1-\gamma_1^2/4} & \text{if $\gamma_2 = 0$ and $\gamma_1 \in 4 \Zz$,} \\ 0 & \text{otherwise}.\end{array} \right.\]
For this formula, we used the identities $\int_{0}^{2\pi} e^{\imath \gamma_2 \theta} \mathrm{d} \theta = 2 \pi \delta_{\gamma_2,0}$ and
$\int_0^{1} T_{\gamma_1} (r) r \mathrm{d} r = \frac12 \frac{1}{1-\gamma_1^2/4}$ if $\gamma_1$ is an element of $4 \Zz$ and zero otherwise.
With the expansion \eqref{20170303146} of $P^{(\vect{m})}_{f}$ in the space $\Pimsquare$, we obtain the Clenshaw-Curtis quadrature formula
\[ Q(f) = \frac{1}{\pi} \int_0^{2\pi} \int_0^{1} P^{(\vect{m})}_{f}(r,\theta) r \, \mathrm{d} r \mathrm{d} \theta = \sum_{k = 0}^{\lfloor
m_1 /2 \rfloor} \frac{c_{(4k,0)}(f)}{1 - 4 k^2}.  \]
The coefficients $c_{(4k,0)}(f)$ on the right hand side depend only on the data values $f(\vect{i})$, $\vect{i} \in \Imm$, and, therefore by \eqref{eq:186181822}, on the function samples given at the rhodonea nodes $\LSm$. The quadrature rule $Q(f)$ is exact for all functions in $\Pimsquare$. Since $c_{(k,0)}(f) = c_{\mathcal{R},(k,0)}(f)$, the same formula holds also true using $P^{(\vect{m})}_{\mathcal{R},f}$ as an interpolation function. The quadrature formula 
$Q(f)$ remains also the same if we use the triangular spectral index set $\Gamtriangle$ instead of $\Gamsquare$.

\begin{figure}[!htb]
\centering
	
\subfigure[\scriptsize Interpolant $P_{\mathcal{R},f}^{(10,11)}$ using 
	$\vect{\Gamma}_{\triangle}^{(10,11)}$. \newline
	$\| P_{\mathcal{R},f}^{(10,11)}\! -\!f \|_\infty \! \approx  \! 0.29348027296549 $ \newline \vspace{1mm}
	$Q(f) \approx  0.03901168892218$ \newline \vspace{1mm}
	$\frac{|Q(f) - I(f)|}{|I(f)|} \approx  0.02355870104964$.
	]{\includegraphics[scale=0.14]{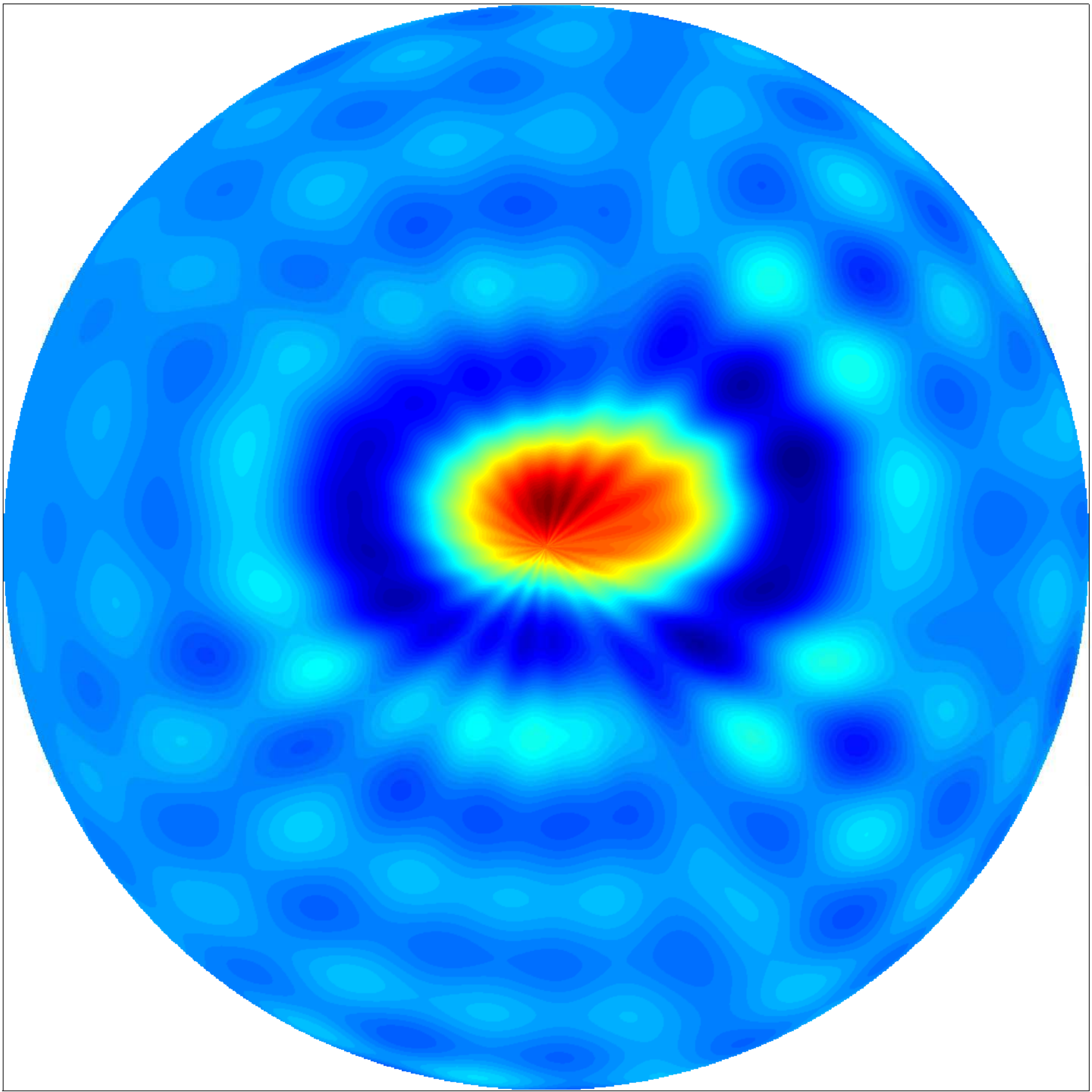}}
\hfill	
\subfigure[\scriptsize Interpolant $P_{\mathcal{R},f}^{(20,21)}$ using 
	$\vect{\Gamma}_{\triangle}^{(20,21)}$. \newline
	$\| P_{\mathcal{R},f}^{(20,21)}\! -\!f \|_\infty \! \approx \! 0.01459069689457$ \newline \vspace{1mm}
	$Q(f) \approx  0.03811412971653 $ \newline \vspace{1mm}
	$\frac{|Q(f) - I(f)|}{|I(f)|} \approx  0.00000923267162$.
	]{\includegraphics[scale=0.14]{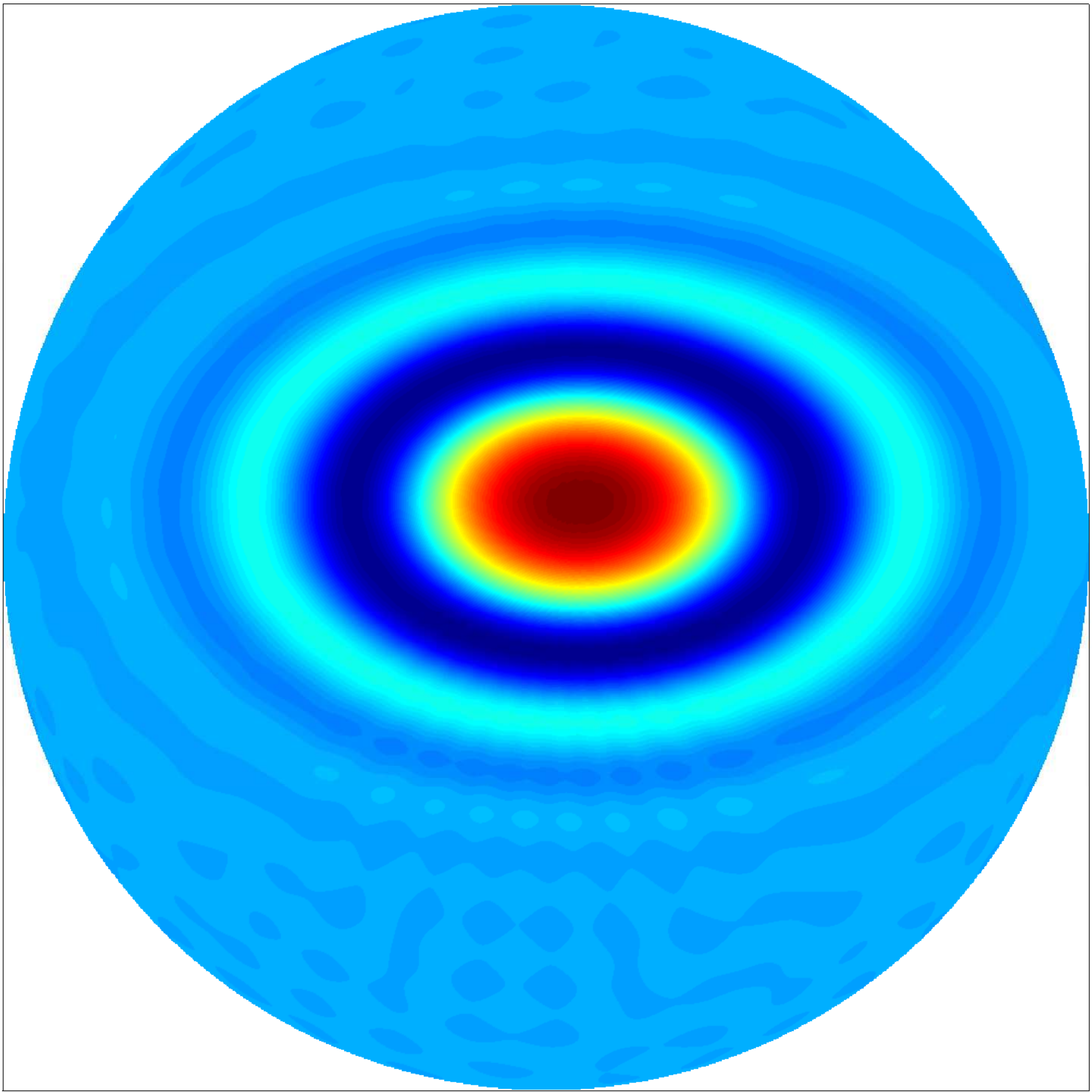}}
\hfill
\subfigure[\scriptsize Interpolant $P_{\mathcal{R},f}^{(30,31)}$ using 
	$\vect{\Gamma}_{\triangle}^{(30,31)}$. \newline
	$\| P_{\mathcal{R},f}^{(30,31)}\! -\!f \|_\infty \! \approx  \! 0.00003902453899$ \newline \vspace{1mm}
	$Q(f) \approx  0.03811377781358$ \newline \vspace{1mm}
	$\frac{|Q(f) - I(f)|}{|I(f)|} \approx  0.00000000028748$
	]{\includegraphics[scale=0.139]{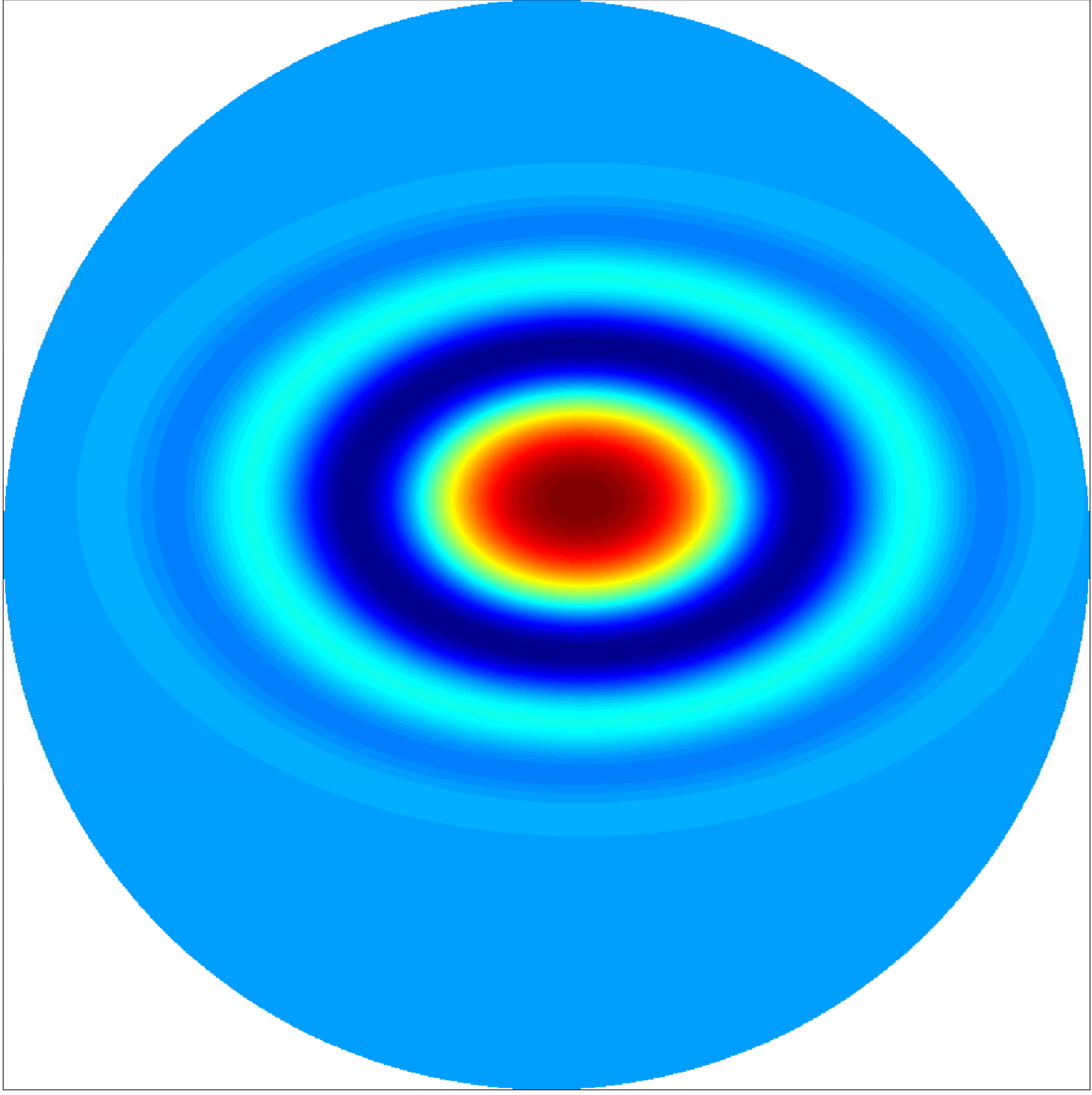}}

\subfigure[\scriptsize Interpolant $P_{\mathcal{R},f}^{(10,11)}$ using 
	$\vect{\Gamma}_{\square}^{(10,11)}$. \newline
	$\| P_{\mathcal{R},f}^{(10,11)}\! -\!f \|_\infty \! \approx  \! 0.28652455823358$ \newline \vspace{1mm}
	$Q(f) \approx  0.03901168892218$ \newline \vspace{1mm}
	$\frac{|Q(f) - I(f)|}{|I(f)|} \approx  0.02355870104964$.
	]{\includegraphics[scale=0.14]{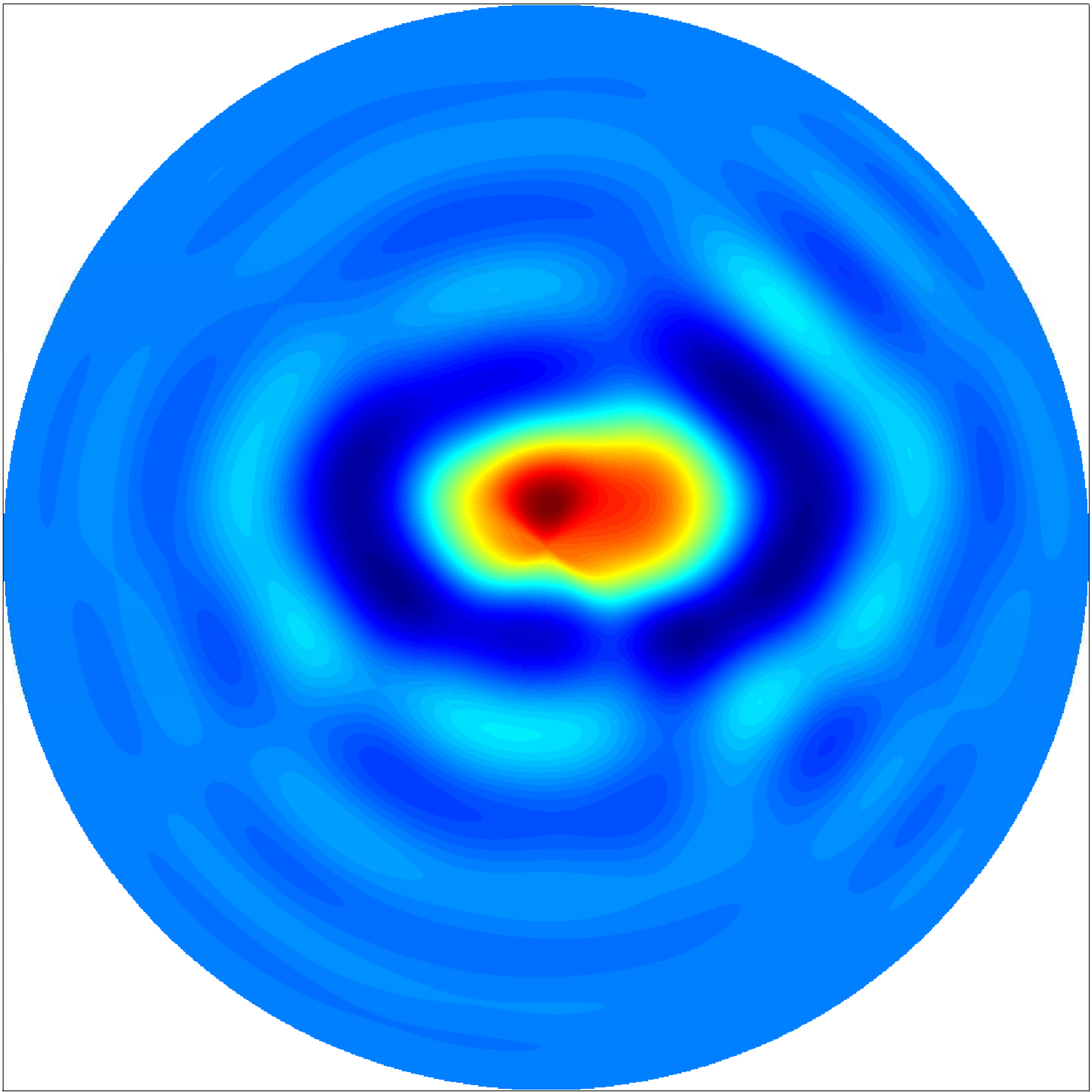}}
\hfill	
\subfigure[\scriptsize Interpolant $P_{\mathcal{R},f}^{(20,21)}$ using 
	$\vect{\Gamma}_{\square}^{(20,21)}$. \newline
	$\| P_{\mathcal{R},f}^{(20,21)}\! -\!f \|_\infty \! \approx \! 0.00410290500954$ \newline \vspace{1mm}
	$Q(f) \approx  0.03811412971653 $ \newline \vspace{1mm}
	$\frac{|Q(f) - I(f)|}{|I(f)|} \approx  0.00000923267162$.
	]{\includegraphics[scale=0.14]{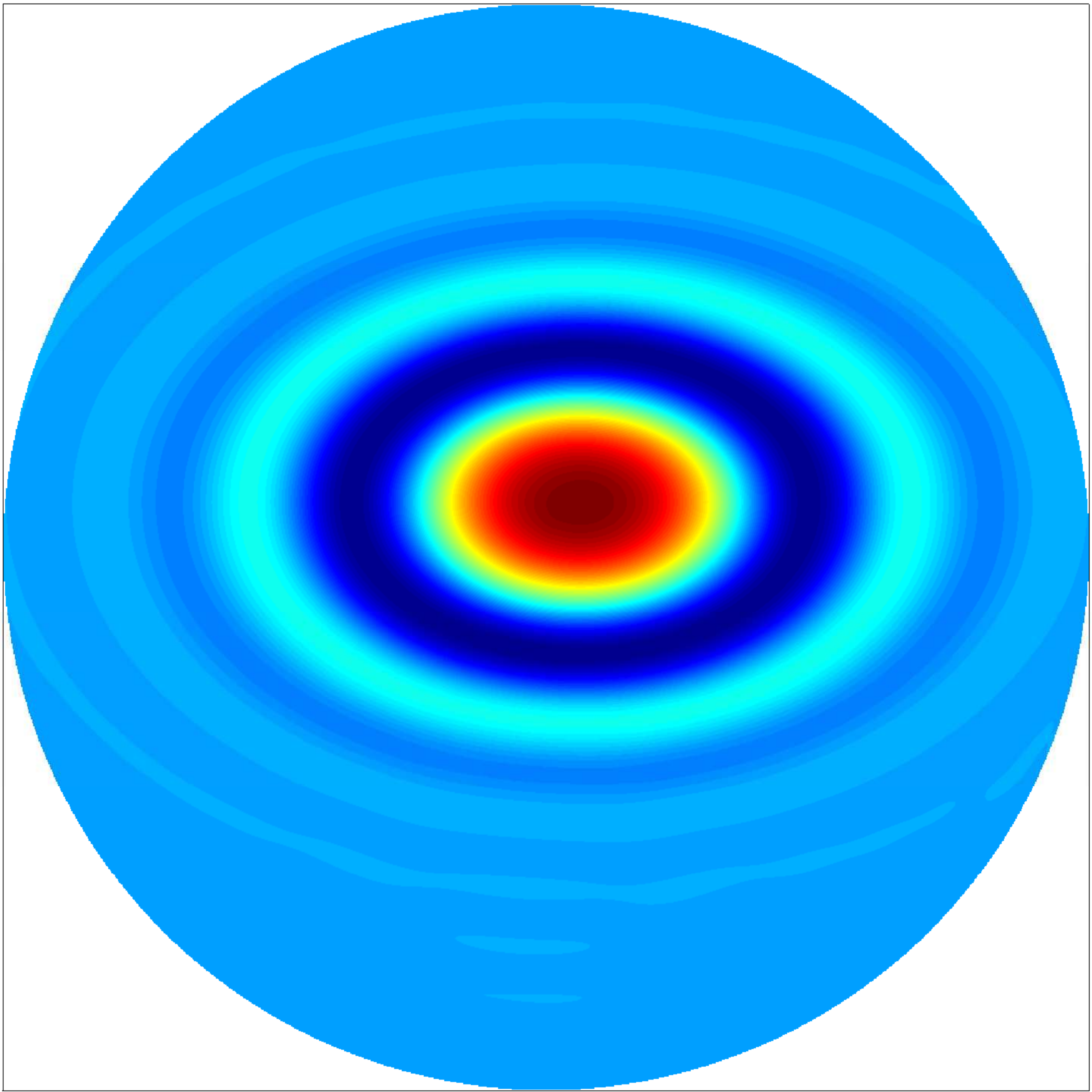}}
\hfill
\subfigure[\scriptsize Interpolant $P_{\mathcal{R},f}^{(30,31)}$ using 
	$\vect{\Gamma}_{\square}^{(30,31)}$. \newline
	$\| P_{\mathcal{R},f}^{(30,31)}\! -\!f \|_\infty \! \approx  \! 0.00004734909880$ \newline \vspace{1mm}
	$Q(f) \approx  0.03811377781358$ \newline \vspace{1mm}
	$\frac{|Q(f) - I(f)|}{|I(f)|} \approx  0.00000000028748$
	]{\includegraphics[scale=0.139]{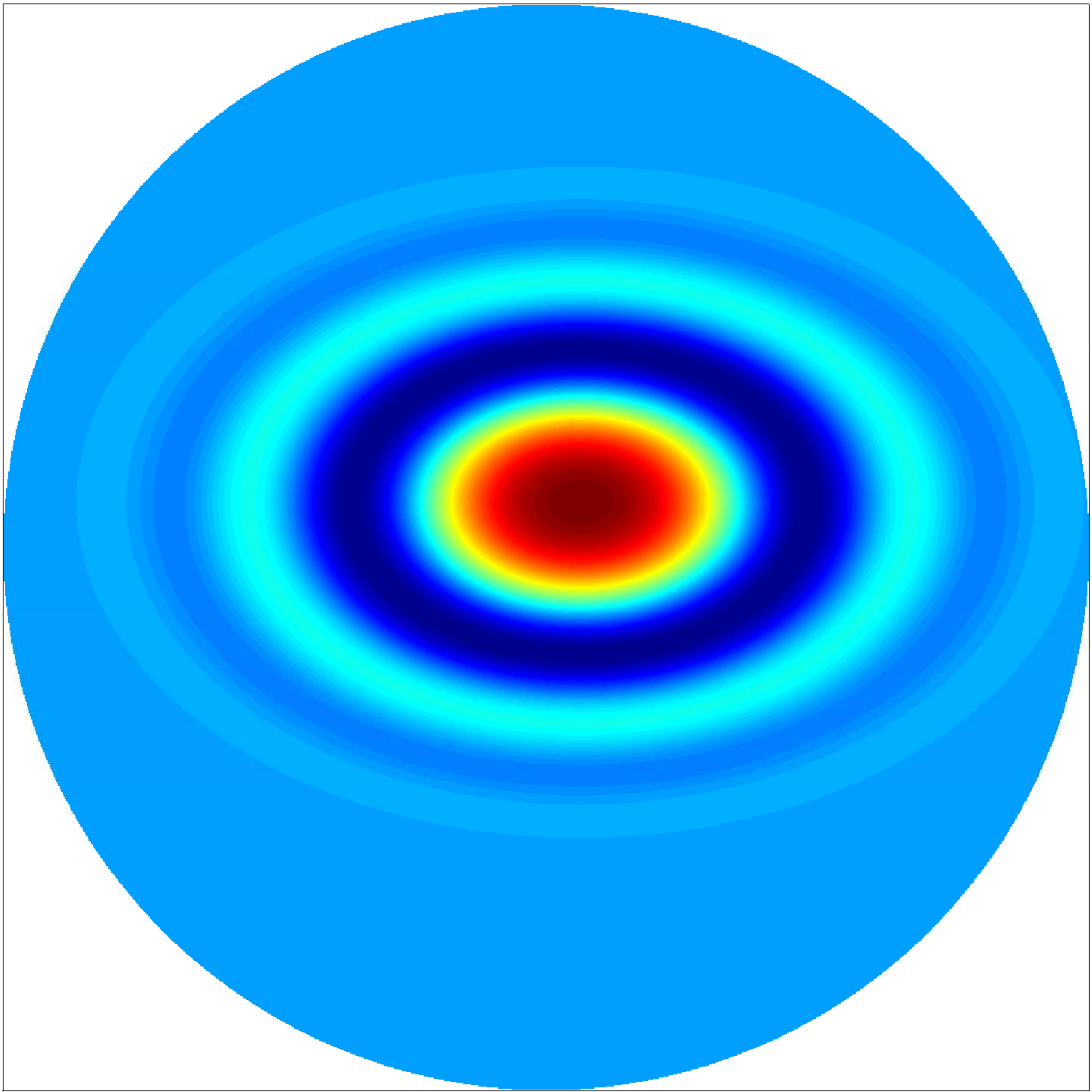}}

  	\caption{Evaluation of the spectral interpolation scheme and the Clenshaw-Curtis quadrature for the function $f$ given in \eqref{eq:testfunction} at 
  	the rhodonea nodes $\LSm$ with different frequency parameters $\vect{m}$. The value $I(f)$ ($\approx 0.03811377782454$) denotes the exact integral value of $f$ over the unit disk. 
  	} \label{fig:LS-4}
\end{figure}

\paragraph{\bf \ref{sec:convergence}.5. A numerical example}
As a final numerical experiment we test the developed interpolation scheme and the Clenshaw-Curtis 
quadrature formula for the function 
\begin{equation} \label{eq:testfunction}
f(\vect{x}) = e^{- 2 ((1.6 x_1-0.1)^2+(2.4 x_2-0.2)^2)} \cos((4x_1-0.25)^2+(6x_2-0.5)^2).
\end{equation}
The results of this test for different frequency parameters $\vect{m} = (m,m+1)$, $m \in \Nn$, are illustrated in Figure \ref{fig:LS-4}. We observe a fast convergence of the interpolant $P_{\mathcal{R},f}^{(m,m+1)}$ and the quadrature value $Q(f)$ towards $f$ and $I(f)$, respectively, as the parameter $m \in \Nn$ gets large. 
This fast spectral convergence is not surprising since $f$ is analytic. The error estimate \eqref{eq:1806211617} provides a convergence rate of the interpolation scheme faster than any polynomial. 

We also compare the interpolation scheme for the two spectral index sets $\Gamsquare$ and $\Gamtriangle$. For smaller values of $m$ the discontinuity of the interpolant in the space $\Pimrealtriangle$ is clearly visible at the center $(0,0)$ of $\Dd$ while, according to Theorem \ref{th:201807231722} and Remark \ref{rem:AB1}, the considered interpolants $P^{(\vect{m})}_{\mathcal{R},f}$ have no discontinuities in $\Pimrealsquare$. For increasing values
of $m$ the differences between the interpolant $P^{(\vect{m})}_{\mathcal{R},f}$ in $\Pimrealtriangle$ and
$\Pimrealsquare$ almost vanish. Since the Clenshaw-Curtis quadrature formula $Q(f)$ is the same for the interpolation spaces $\Pimrealsquare$ and $\Pimrealtriangle$ we observe no differences in the evaluations of $Q(f)$ for the two spectral index sets. A {\sc Matlab} 
code for this numerical example with an implementation of the spectral interpolation scheme on the rhodonea nodes can be found at 
{\it https://github.com/WolfgangErb/RDisk}.

\section{Proofs} \label{sec:proof}

\paragraph{\bf 9.1. Proof of Proposition \ref{prop-11}}

\begin{proofof}{Proposition \ref{prop-11}}
For $s,t \in \Rr$, we write $s\eqsim t$ if $s$ and $t$ satisfy the equivalence relation $t-s\in 2\pi\mathbb{Z}$. In a first step, we determine all points $t \in [0,2\pi)$ so
that $\rhodonea(t)=(0,0)$, i.e., $\rhodonea(t)$ is the center of the unit disk. By the definition \eqref{201509161237}
of the rhodonea curve $\rhodonea$, we have $\rhodonea(t)=(0,0)$
if and only if
\begin{equation}\label{170210-1}
m_2 t \eqsim \pm \pi/2,
\end{equation}
i.e., if and only if $t = t^{(\vect{m})}_{l}$ for some $l\in \{m_1, 3m_1, \ldots ,(4m_2-1)m_1\}$. This provides the first statements (i) and (i)' of Proposition \ref{prop-11}.

In a second step, we consider now for fixed $t \in [0,2\pi)$ the case $\rhodonea(t) \neq (0,0)$. By the definition \eqref{201509161237} of $\rhodonea$, we get $s\in \mathcal{S}(t)$ if and only if
\[\cos(m_2 s) = v \cos(m_2 t) \quad \text{and} \quad  \left( \begin{array}{l} \cos(m_1 s - \alpha \pi) \\ \sin(m_1 s - \alpha \pi) \end{array}  \right) =  v  
 \left( \begin{array}{l} \cos(m_1 t - \alpha \pi) \\ \sin(m_1 t - \alpha \pi) \end{array}  \right),\]
for some $v \in \{-1,1\}$. In the left equation, we have equality exactly if 
$m_2 s  \eqsim u m_2 t + \frac{1-v}{2} \pi$. In the right equation, equality is obtained if
$m_1 s \eqsim m_1 t + \frac{1-v}{2} \pi$ holds true. Combining these observations, we get
$s\in \mathcal{S}(t)$ if and only if
\begin{equation}\label{170210-2}
m_2( s - u t) + \frac{1-v}{2} \pi \eqsim 0 \quad \text{and} \quad m_1 ( s - t) + \frac{1-v}{2} \pi \eqsim 0 \qquad \text{for some $u,v\in \{-1,1\}$}.
\end{equation}
We characterize now all $s \in [0,2\pi)$ that satisfy the conditions in \eqref{170210-2}.
As $m_1$ and $m_2$ are relatively prime, B\'ezout's lemma provides two integers
$a, b\in\mathbb{Z}$ such that $a m_1 + b m_2 = 1$. 
Multiplying the first and the second identity in \eqref{170210-2} with $b$ and $a$, respectively, and adding them up, we obtain 
\begin{equation} \label{170210}
s \eqsim t - b m_2 (1-u) t - (a+b) \frac{1-v}{2} \pi
\end{equation}
as a description for all possible solutions $s$ of \eqref{170210-2}. 
In particular,
this implies that $s \eqsim t$ or $s \eqsim t + \pi$ are the only possible solutions for $u = 1$.
If $u = -1$, we multiply the first and the second identity in \eqref{170210-2} with $m_1$ and $m_2$, respectively, and obtain
\begin{align} \notag 2 m_1 m_2 s\eqsim  m_1 m_2 (s + t) + m_2 m_1 ( s - t) \eqsim (m_1 + m_2) \frac{1-v}{2} \pi, \\
2 m_1 m_2 t \eqsim  m_1 m_2 (s + t) - m_2 m_1 ( s - t)  \eqsim (m_1 - m_2) \frac{1-v}{2} \pi. \label{eq:00101} \end{align}
Therefore, for $u = -1$ we can conclude that $s = t_{l'}^{(\vect{m})}$ and $t = t_{l}^{(\vect{m})}$ for some $l,l' \in \Zz$. \\[-3mm]

Based on these deductions we can now derive the remaining properties. We distinguish the two cases 
$m_1 + m_2$ odd and $m_1 + m_2$ even:

If we suppose that $m_1 + m_2$ is odd and $\rhodonea(t) \neq (0,0)$, $s \eqsim t + \pi$ can not satisfy both identities in \eqref{170210-2}, and is therefore not a solution of \eqref{170210-2}. Hence, if
$t \neq t_{l}^{(\vect{m})}$ for some $l \in \{0, \ldots, 4 m_1 m_2 -1\}$, then
$t$ is the only element of $[0,2 \pi)$ in $\mathcal{S}^{(\vect{m})}(t)$ (corresponding to the sole solution of \eqref{170210-2} with the
values $u = 1$ and $v = 1$)
and the largest part of statement (iii) is proven. \\
If $t \in [0, 2\pi)$ and $t = t_{l}^{(\vect{m})}$ for some $l\in \{0,\ldots,4m_1m_2-1\}$, then further solutions of \eqref{170210-2} for
the value $u = -1$ are possible. The corresponding possibilities given by \eqref{170210} are $s \eqsim t - 2 b m_2 t$ or $s' \eqsim s + \pi$.
Since we are in the case $m_1+m_2$ odd, only one
of the two solutions $s$ and $s'$ is possible. Also, $s$ or $s'$ do not depend on the particular
value of the integer $b$ from B\'ezout's lemma (since $b$ is uniquely determined modulo $m_1$). The identity \eqref{170210} therefore gives in this case
exactly one solution of \eqref{170210-2} for $u = -1$ (we denote this solution as $s \in [0,2\pi)$). If $t = t_{l}^{(\vect{m})}$ with $l \equiv 0 \mod 2 m_1$
we obtain $s \eqsim t$ and $s \not\eqsim t$ if $l \not \equiv 0 \mod m_1$. This yields the remaining assertions (ii) and (iii) of the proposition.
Note that the case $l \equiv m_1 \mod 2 m_1$ is already treated in (i).

Finally, we shortly discuss the case when both integers $m_1$ and $m_2$ are odd, or in other words, when $m_1 + m_2$ is even. The statements (ii)' and (iii)' can be deduced in a similar
way as before with one cardinal difference: in this case the curve $\rhodonea(t)$ is traversed twice as $t$ varies from $0$ to $2\pi$.
In \eqref{170210-2}, we consequently see that if $s \in \mathcal{S}^{(\vect{m})}(t)$ then also $s + \pi \in \mathcal{S}^{(\vect{m})}(t)$.
This yields additional solutions which double the value of $\# \mathcal{S}^{(\vect{m})}(t)$ in
(ii)' and (iii)'. Further, from \eqref{eq:00101} we can deduce that the double points of the curve are given at the positions
$t = t_{2l}^{(\vect{m})}$ for some $l \in \Zz$. \qed
\end{proofof}

\paragraph{\bf 9.2. Proof of Theorem \ref{cor-111}}

To prove Theorem \ref{cor-111}, it is necessary to establish a relation between the nodal index sets $\Imm$ 
and the sampling points along the rhodonea curves $\rhodonea$. This relation can be 
extracted from the following auxiliary result:

\begin{proposition}\label{1509221521} Let $\vect{m} \in \Nn^2$ and $g = \gcd(\vect{m})$
For all $l \in \{0,\ldots, 4 m_{1} m_{2}/g -1\}$ and  $\rho \in \{0,\ldots,2 g-1\}$, there exists an $\vect{i}\in \Imm$ and $u,v \in\{-1,1\}$ such that
\begin{align}
i_{1} &\equiv u (v l + (1-v) m_1) \mod 4 m_1, \label{1509221526} \\
i_{2} &\equiv l - 2\rho - (1-v) m_2 \mod 4 m_2. \label{1509221526B}
\end{align}
The index $\vect{i}\in \Imm$ is uniquely determined by \eqref{1509221526} and \eqref{1509221526B} and provides a well-defined surjective mapping $\vect{i}^{(\vect{m})}: \{0,\ldots, 4 m_{1} m_{2}/g -1\} \times \{0,\ldots,2 g-1\} \to \Imm$ by $\vect{i}^{(\vect{m})}(l,\rho)=\vect{i}$. 
Further, $\vect{i}^{(\vect{m})}(l,\rho)\in \Imm_{0}$ ( $\vect{i}^{(\vect{m})}(l,\rho)\in \Imm_{1}$) holds true if and only if $l$ is even ($l$ is odd).\\[-3mm]

\noindent With the additional convention
\begin{equation} \label{1509221526C}
\text{$u = 1$ if $l \equiv 0 \mod 4 m_1$,\hspace{1cm} $u = -1$ if $l \equiv 2 m_1 \mod 4 m_1$,}
\end{equation}
the numbers $u,v \in \{-1,1\}$ are uniquely determined by \eqref{1509221526} and \eqref{1509221526B}.
This gives for $\vect{i}\in \Imm$ the cardinalities 
\begin{equation} \label{1509221526D} \#\{\,(l,\rho) \,|\,\vect{i}^{(\vect{m})}(l,\rho)=\vect{i}\,\}=\left\{ \begin{array}{ll} 4 \quad & \text{if \ $0 < i_1 \leq m_1$,} \\
2 \quad & \text{if \ $i_1 = 0$.} \end{array}\right.\end{equation}
\end{proposition}

\begin{proof}
For $l \in \{0,\ldots, 4 m_{1} m_{2}/g -1\}$ we can find an integer $0\leq i_{1}\leq m_{1}$ and $u, v \in \{-1,1\}$ such 
that 
$u i_1 + (1-v) m_1 \equiv l \mod 4 m_1$
holds true, i.e. that \eqref{1509221526} is satisfied. The number $i_1$ in this equation is uniquely determined by $l$, whereas, 
with the convention \eqref{1509221526C}, the numbers $u$ and $v$ are uniquely determined by \eqref{1509221526} exactly if $l \not \equiv m_1 \mod 4 m_1$ and
$l \not \equiv - m_1 \mod 4 m_1$.
In this case, the tuple $(l,\rho)$ and the number $v$ given by \eqref{1509221526} determine an unique integer
$-2m_2 < i_{2} \leq 2m_{2}$ such that \eqref{1509221526B} is satisfied. In the remaining case when $l \equiv m_1 \mod 4 m_1$ or $l \equiv - m_1 \mod 4 m_1$,
the tuple $(l,\rho)$ yields an unique $-2m_2 < i_{2} \leq 0$ and $v \in \{-1,1\}$ such that \eqref{1509221526B} is satisfied. Furthermore, in this
case, the so determined $v \in \{-1,1\}$ fixes also the number $u \in \{-1,1\}$ in \eqref{1509221526}. Since $i_{1} \equiv l\equiv  i_{2} \mod 2$, 
we can finally state that the index $\vect{i}$ determined in this way from \eqref{1509221526} and \eqref{1509221526B} is an element of $\Imm$. Further, looking at
the definition in \eqref{eq:0901} we also see that $\vect{i}^{(\vect{m})}(l,\rho)\in \Imm_{0}$ or $\vect{i}^{(\vect{m})}(l,\rho)\in \Imm_{1}$ 
holds if and only if $l\equiv 0 \mod 2$ or $l\equiv 1 \mod 2$, respectively. 

We finally prove \eqref{1509221526D}. Let $\vect{i}\in\Imm$ and $u,v\in \{-1,1\}$. Then for
$a_{1} = u (v i_1 + (1-v) m_1)$ there is a uniquely determined $\rho\in \{0,\ldots,2 g-1\}$ such
that $a_2 = i_2 + (1-v) m_2 + 2 \rho$ satisfies $a_1 \equiv a_2 \mod 4g$. The Chinese remainder theorem  now yields a unique number
$l\in \{0,\ldots,4m_1m_2/g-1\}$ satisfying
\[
 a_1 \equiv l \mod 4m_1, \qquad a_2 \equiv l \mod 4m_2.
\]
With the convention \eqref{1509221526C}
the numbers $a_1$ and $a_2$ are uniquely determined by $\vect{i}\in\Imm$ and $u,v\in \{-1,1\}$. Thus,
the numbers $(l,\rho)$ satisfying \eqref{1509221526} and \eqref{1509221526B} are uniquely determined by $\vect{i}\in\Imm$ and $u,v\in \{-1,1\}$.
In the case $0 < i_1 = m_1$
both choices of $u$ and $v$ give
distinct elements $(l,\rho)$, whereas, according to the convention \eqref{1509221526C}, in the case $i_1 = 0$ only the parameter $v$ can
be chosen freely. \qed
\end{proof}

\begin{proofof}{Theorem \ref{cor-111}}
The definitions \eqref{201509161237} and \eqref{eq-samples1} of the rhodonea curve $\rhodonea$ and the sampling points $t^{(\vect{m})}_{l}$ give us directly the identity
\[ \ts \vect{\rho}^{(\vect{m})}_{\rho/m_2} (t^{(\vect{m})}_{ l}) =
\left( \cos \left( \frac{ l \pi}{2 m_1} \right) \cos \left( \frac{l \pi}{2 m_2} - \frac{2 \rho}{2 m_2} \pi \right), \
\cos \left(\frac{ l \pi}{2 m_1} \right)
\sin \left(\frac{l \pi}{2 m_2} - \frac{2 \rho}{2 m_2} \pi \right) \right).\]
Now, by Proposition \ref{1509221521}, we can find an index $\vect{i}\in \Imm$ and $u,v \in\{-1,1\}$ such that \eqref{1509221526} and \eqref{1509221526B} are satisfied. This implies
\begin{align*} \vect{\rho}^{(\vect{m})}_{\rho/m_2} (t^{(\vect{m})}_{ l}) &=  \ts
\left( \cos \left( \frac{ v i_1 \pi}{2m_1} + \frac{1-v}{2} \pi \right) \cos \left( \frac{ i_2 \pi}{2m_2} + \frac{1-v}{2}\pi \right), \
\cos \left( \frac{ v i_1 \pi}{2m_1} + \frac{1-v}{2} \pi \right) \sin \left( \frac{ i_2 \pi}{2m_2} + \frac{1-v}{2} \pi \right)  \right) \\
&=  \ts
\left( \cos \left( \frac{ i_1 \pi}{2m_1} \right) \cos \left( \frac{ i_2 \pi}{2m_2} \right), \ \cos \left( \frac{ i_1 \pi}{2m_1} \right)
\sin \left(\frac{ i_2 \pi}{2m_2} \pi \right)  \right) = \vect{x}_{\vect{i}}^{(\vect{m})}.
\end{align*}
The reverse implication is obtained by inverting these steps:
for given $\vect{x}_{\vect{i}}^{(\vect{m})}$, we can fix $u,v \in\{-1,1\}$ and Proposition \ref{1509221521} yields a unique tuple $(l,\rho)$ with
$\vect{x}_{\vect{i}}^{(\vect{m})} = \vect{\rho}^{(\vect{m})}_{\rho/m_2} (t^{(\vect{m})}_{ l})$. \qed
\end{proofof}

\paragraph{\bf 9.3. Proof of Theorem \ref{thm:decompositionrhodonea}}

\begin{proofof}{Theorem \ref{thm:decompositionrhodonea}} (a)
Let $\vect{x}(r,\theta)\in \mathcal{R}^{(\vect{m})}$. We choose $s \in [0,\pi/2]$ and $t'\in \mathbb{R}$ such that 
$r=\cos(s)$ and $T_{m_{1}}(r) = \cos (m_1 s) = \cos (m_2 \theta) = \cos(\frac{m_1 m_2}{g} t')$ . 
Then, we can find $v_1,v_2 \in \{-1,1\}$ and $h_1,h_2 \in\mathbb{Z}$ such that 
\begin{align*}
m_1 s &= v_1 \left(\frac{m_1 m_2}{g} t' + h_1 \pi\right), \qquad m_2 \theta = v_2 \left(\frac{m_1 m_2}{g} t' + h_2 \pi\right), 
\end{align*}
and therefore 
\[r = \cos \left(\frac{m_2}{g} t' + \frac{h_1}{m_1} \pi\right) \quad \text{and} \quad \theta = v_2 \left(\frac{m_1}{g} t' + \frac{h_2}{m_2} \pi \right).\]
Further, there is a unique $\rho \in \{0, \ldots, 2 g -1\}$ such that $h_1 \equiv h_2 + v_2 \rho \mod 2g$. Then, by the Chinese remainder theorem we can 
find an $l \in \Zz$ such that $l\equiv h_1 \mod 2 m_1$ and $l\equiv h_2 + v_2 \rho \mod 2 m_1$. This gives (we assume that the angle $\theta$ is an element in $\Rr / (2 \pi \Zz)$)
\[r = \cos \left(\frac{m_2}{g} t' + \frac{l}{m_1} \pi \right) \quad \text{and} \quad \theta = v_2 \left(\frac{m_1}{g} t' + \frac{l- v_2\rho}{m_2} \pi \right).\]
Then, introducing  $t= v_2(t'+l g \pi/ (m_1 m_2))$ we obtain
\[r = \cos \left(\frac{m_2}{g} t \right) \quad \text{and} \quad \theta = \frac{m_1}{g} t - \frac{\rho}{m_2} \pi.\] 
Therefore $\vect{x} \in \vect{\varrho}^{(\vect{m})}_{\rho/m_2}([0,P))$ for some $\rho \in \{0, \ldots, 2g-1\}$.
The implication $\vect{\varrho}^{(\vect{m})}_{\rho/m_2} \subseteq \mathcal{R}^{(\vect{m})}$ is easily verified
by inserting the curve in the description \eqref{1509222011} of the rhodonea variety $\mathcal{R}^{(\vect{m})}$. 

(b) We have a look at the definition \eqref{eq:0917234} of the points $\vect{x}^{(\vect{m})}_{\vect{i}}$ in $\LSm$. Plugging the points
$\vect{x}^{(\vect{m})}_{\vect{i}}$ into the polar equation \eqref{1509222011} of the rhodonea variety, we obtain 
$T_{m_1}(r^{(m_1)}_{i_1}) = \cos (m_2 \theta^{(m_2)}_{i_2}) = 1$ if $\vect{i} \in \Imm_0$ and $T_{m_1}(r^{(m_1)}_{i_1}) = \cos (m_2 \theta^{(m_2)}_{i_2}) = 0$ if $\vect{i} \in \Imm_1$.
On the other hand, it is well-known that the extrema of the univariate functions $T_{m_1}(r)$ and $\cos (m_2 \theta)$ are attained at $r=r^{(m_1)}_{i_1}$
and $\theta = \theta^{(m_2)}_{i_2}$, $\vect{i} \in \Imm_0$, respectively. Also, it is well-known that all roots of $T_{m_1}(r)$ and $\cos (m_2 \theta)$ are given by
$r=r^{(m_1)}_{i_1}$ and $\theta = \theta^{(m_2)}_{i_2}$, $\vect{i} \in \Imm_1$, respectively. \qed
\end{proofof}

\paragraph{\bf 9.4. Proofs of Section \ref{1507091240}}

The proofs of Section \ref{1507091240} base on the following technical result:

\begin{proposition}\label{1507211320}
Let $\vect{\gamma}\in\Zz^{2}$ and $\gamma_1 \equiv \gamma_2 \mod 2$. If $\int\polbas \mathrm{d}\rule{1pt}{0pt}\mathrm{w}\neq 0$, then
\begin{equation}\label{1507201132}
\text{there exist $(h_1,h_2) \in \Zz^{2}$ with $\gamma_{1}=2 h_{1}m_{1}$, $\gamma_{2}=2h_{2}m_{2}$, and
$h_1 + h_2 \equiv 0 \mod 2$}.
\end{equation}
If \eqref{1507201132} is satisfied, then $\int\polbas \mathrm{d}\rule{1pt}{0pt}\mathrm{w}=1$.
\end{proposition}

In the proof of Proposition \ref{1507211320}, we use the well-known trigonometric identity
\begin{equation}\label{1506171253}
\sum_{l=0}^N \mathrm{e}^{\imath l \vartheta} = \left\{ \begin{array}{ll} \frac{\mathrm{e}^{\imath (N+1) \vartheta}-1}{\mathrm{e}^{\imath \vartheta}-1}
\quad & \vartheta\notin 2\pi\mathbb{Z},\\ N+1 & \vartheta \in 2\pi\mathbb{Z}, \end{array} \right.
\qquad N\in\Nn_0.
\end{equation}

\begin{proof}
Using Proposition \ref{1509221521}, we can manipulate the discrete integral $\int\polbas\mathrm{d}\rule{1pt}{0pt}\mathrm{w}$ as follows:
\begin{align*}
\int\polbas\mathrm{d}\rule{1pt}{0pt}\mathrm{w}
&=  \sum_{\vect{i}\in\Imm}
\mathrm{w}^{(\vect{m})}_{\vect{i}} \cos(\gamma_{1} i_1 \pi/(2m_{1})) \mathrm{e}^{\imath \gamma_{2} i_2 \pi/(2m_{2}) } \\
&=  \frac14 \sum_{u,v \in \{\pm 1\}}\sum_{\vect{i}\in\Imm} \mathrm{w}^{(\vect{m})}_{\vect{i}}
\mathrm{e}^{\imath \,( \gamma_{1} u v i_1/(2m_1) \pi + \gamma_1 (1-v)/2 \pi) + \gamma_{2} i_2/(2m_2) \pi + \gamma_2 (1-v)/2 \pi))}\\
&=  \frac{1}{8m_1m_2} \sum_{l = 0}^{4 m_1 m_2/g} \sum_{\rho = 0}^{2g - 1}
\mathrm{e}^{\imath \,( \gamma_{1} l \pi/(2m_{1}) + \gamma_{2} l \pi/(2m_{2}) + 2 \gamma_2 \rho \pi / (2m_2))}
\end{align*}
The trigonometric identity \eqref{1506171253} implies that the last sum is different from zero if and
only if $\gamma_{1} /(2m_{1}) + \gamma_{2} /(2m_{2}) \in 2 \Zz$ and $\gamma_2 /(2m_2) \in \Zz$
are satisfied. Therefore, if we assume that $\int\polbas\mathrm{d}\rule{1pt}{0pt}\mathrm{w} \neq 0$ then $\gamma_2 = 2 h_2 m_2$ with some
integer $h_2 \in \Zz$ and $\gamma_{1} /(2m_{1}) + \gamma_2 / (2m_2) \in 2 \Zz$. In particular, also $\gamma_1 = 2 h_1 m_1$ with some $h_1 \in \Zz$.
Further, we have $h_1 + h_2 \in 2 \Zz$. This proves the identity \eqref{1507201132}. If \eqref{1507201132} is satisfied then the trigonometric identity \eqref{1506171253} yields
\begin{align*}
\int\polbas\mathrm{d}\rule{1pt}{0pt}\mathrm{w}
&=  \frac{1}{8m_1m_2} \sum_{l = 0}^{4 m_1 m_2 / g} \mathrm{e}^{\imath \,l ( \gamma_{1}  \pi/(2m_{1}) + \gamma_{2} \pi/(2m_{2}))} \sum_{\rho = 0}^{2 g -1}
\mathrm{e}^{\imath \,( 2 \gamma_2 \pi / 2 m_2) \rho } = 1.
\end{align*} \qed
\end{proof}

In order to prove Theorem \ref{1507091911}, i.e., to show that the rectangular set $\Gamsquare$ is a spectral index set, we will use the two identities
\begin{align}
\label{1507222159}
\polbas \overline{\chi^{(\vect{m})}_{\vect{\gamma}'}}
&= \frac{1}{2}  \left(\chi^{(\vect{m})}_{(\gamma_1+\gamma_1',\gamma_2 - \gamma_2')}
+ \chi^{(\vect{m})}_{{(\gamma_1-\gamma_1',\gamma_2 - \gamma_2')}} \right),\\
\overline{\chi^{(\vect{m})}_{\vect{\gamma}}}
&= \chi^{(\vect{m})}_{(\gamma_1,-\gamma_2)}, \label{1507222159C}
\end{align}
which are satisfied for all $\vect{\gamma},\vect{\gamma}' \in \Zz^2$. Formulas \eqref{1507222159}
and \eqref{1507222159C} are a direct consequence of the definition \eqref{A1508291531}
of the discrete functions $\polbas$ as well as the cosine product formula.

\begin{proofof}{Theorem \ref{1507091911}}
We will constantly use condition \eqref{1507201132} and Proposition \ref{1507211320} 
to derive the values of the integrals. For a simpler notation, we denote the index vectors on the right hand side of \eqref{1507222159} by
\[ \vect{\gamma}^+ = (\gamma_1+\gamma_1',\gamma_2-\gamma_2') \quad \text{and} \quad \vect{\gamma}^- = (\gamma_1-\gamma_1',\gamma_2-\gamma_2').\]
We assume first that $\vect{\gamma},\vect{\gamma}'\in \Gamsquare$ and $\vect{\gamma} \neq \vect{\gamma}'$: Since $-2 m_2 < \gamma_2-\gamma_2' < 2 m_2$, the condition
\eqref{1507201132} can be satisfied for $\vect{\gamma}^+$ and $\vect{\gamma}^-$ only if $\gamma_2 = \gamma_2'$. Since we assume that $\vect{\gamma} \neq \vect{\gamma}'$, we get in this case $\gamma_1 \neq \gamma_1'$. This, on the other hand, implies that $\gamma_1-\gamma_1' \in 4 \Zz$, $\gamma_1 + \gamma_1' \in 4 \Zz$ is not possible, and 
therefore that $\vect{\gamma}^+$ and $\vect{\gamma}^-$ can not satisfy the condition \eqref{1507201132}. The product formula \eqref{1507222159} now yields
the orthogonality $\int\polbas \overline{\chi^{(\vect{m})}_{\vect{\gamma}'}} \mathrm{d}\rule{1pt}{0pt}\mathrm{w} = 0$. \\[2mm]
Now, consider $\vect{\gamma},\vect{\gamma}'\in \Gamsquare$ and $\vect{\gamma} = \vect{\gamma}'$:
In this case, we have $\vect{\gamma}^+ = (2 \gamma_1,0)$ and $\vect{\gamma}^- = (0,0)$.
Since $\vect{\gamma}\in \Gamsquare$, we have $0 \leq 2 \gamma_1 \leq 4 m_1$. Therefore, condition \eqref{1507201132} is always satisfied for
$\vect{\gamma}^-$ and satisfied for $\vect{\gamma}^+$ precisely if $\gamma_1 \in \{0,2m_1\}$. Proposition \ref{1507211320} therefore implies \eqref{1508221825}. \\[2mm]
Finally, since the functions $\polbas$, $\vect{\gamma} \in \Gamsquare$, are pairwise orthogonal, 
they are in particular linearly independent and span a subspace of dimension $\# \Gamsquare = 2 m_1 m_2 + m_2$ in
$\mathcal{L}(\Imm)$. Since $\dim \funpol(\Imm) = \# \Imm = (2 m_1 +1) m_2$ is of the same 
complexity, this subspace coincides with $\funpol(\Imm)$. \qed
\end{proofof}

\begin{proofof}{Theorem \ref{1507091912}} The functions $\polbasreal$ are real and satisfy
\begin{equation*}
\polbasreal = \left\{  \begin{array}{ll}
\Re \chi^{(\vect{m})}_{(\vect{\gamma})} = \frac12
(\chi^{(\vect{m})}_{\vect{\gamma}} + \overline{\chi^{(\vect{m})}_{\vect{\gamma}}}),
                                        & \text{if $\vect{\gamma}$ is in the sets (i) or (iii) of \eqref{1702291124}}, \\
\Im \chi^{(\vect{m})}_{(\vect{\gamma})} = \frac1{2 \imath}
(\chi^{(\vect{m})}_{\vect{\gamma}} - \overline{\chi^{(\vect{m})}_{\vect{\gamma}}}),
                                        & \text{if $\vect{\gamma}$ is in the sets (ii) or (iv) of \eqref{1702291124}}.
\end{array} \right.
\end{equation*}
Using the condition \eqref{1507201132} of Proposition \ref{1507211320} in combination with the trigonometric identities \eqref{1507222159} and \eqref{1507222159C}, the orthogonality of the basis functions can be derived in the same way as in the proof of Theorem \ref{1507091911}. In the following, we provide the calculation of 
the norms $\|\polbasreal\|_{\mathrm{w}}^2$. Using \eqref{1507222159} and \eqref{1507222159C} for $\polbasreal = \Re \chi^{(\vect{m})}_{(\vect{\gamma})}$, we get
\begin{align*} \|\polbasreal\|_{\mathrm{w}}^2 &= \frac14 \int \left|\chi^{(\vect{m})}_{(\gamma_1, \gamma_2)}
+ \chi^{(\vect{m})}_{(\gamma_1,- \gamma_2)}\right|^2 \mathrm{d}\rule{1pt}{0pt}\mathrm{w} \\
&= \frac{1}{8} \int \left( 2 \chi^{(\vect{m})}_{(0,0)} + 2 \chi^{(\vect{m})}_{(2 \gamma_1,0)} + \chi^{(\vect{m})}_{(2\gamma_1, 2\gamma_2)} + \chi^{(\vect{m})}_{(0, 2\gamma_2)}
+ \chi^{(\vect{m})}_{(2\gamma_1, - 2\gamma_2)} + \chi^{(\vect{m})}_{(0, -2\gamma_2)} \right) \mathrm{d}\rule{1pt}{0pt}\mathrm{w}.
\end{align*}
On the other hand, if $\polbasreal = \Im \chi^{(\vect{m})}_{(\vect{\gamma})}$, we have
\begin{align*} \|\polbasreal\|_{\mathrm{w}}^2 &= \frac14 \int \left|\chi^{(\vect{m})}_{(\gamma_1, \gamma_2)}
- \chi^{(\vect{m})}_{(\gamma_1,- \gamma_2)}\right|^2 \mathrm{d}\rule{1pt}{0pt}\mathrm{w} \\
&= \frac{1}{8} \int \left( 2 \chi^{(\vect{m})}_{(0,0)} + 2 \chi^{(\vect{m})}_{(2 \gamma_1,0)} - \chi^{(\vect{m})}_{(2\gamma_1, 2\gamma_2)} - \chi^{(\vect{m})}_{(0, 2\gamma_2)}
- \chi^{(\vect{m})}_{(2\gamma_1, - 2\gamma_2)} - \chi^{(\vect{m})}_{(0, -2\gamma_2)} \right) \mathrm{d}\rule{1pt}{0pt}\mathrm{w}.
\end{align*}
In both cases, using Proposition \ref{1507211320}, we can explicitly evaluate the integrals on the right hand side. 
Depending on the different cases given in \eqref{1508221826}, the corresponding values for the norm in \eqref{1508221826} can be obtained directly. \qed
\end{proofof}

\paragraph{\bf 9.5. Proofs of Section \ref{sec:interpolation}}

\begin{proofof}{Theorem \ref{201512131945} and Theorem \ref{201512131946}} 
The proof of the two theorems differs only in the choice of the basis system. We will therefore restrict our attention to Theorem \ref{201512131945}. 

We denote by $\delta_{\vect{j}}(\vect{i}) = \delta_{\vect{i}\vect{j}}$, $\vect{j} \in \Imm$, the set of all Dirac functions on $\Imm$. They clearly form an orthogonal basis of $\funpol(\Imm)$.
By Definition \ref{def:spectralindex} of the spectral index set $\Gam$, also the function system $\polbas$, $\vect{\gamma} \in \Gam$, is an orthogonal basis for the space $\funpol(\Imm)$.  
We can therefore expand the Dirac functions $\delta_{\vect{j}}$, $\vect{j} \in \Imm$, as
\[\delta_{\vect{j}}(\vect{i}) = \sum_{\vect{\gamma} \in \Gam} \frac{\langle\;\! \delta_{\vect{j}},\polbas \rangle_{\mathrm{w}}}{\|\polbas\|_{\mathrm{w}}^2} \polbas(\vect{i}) =
\mathrm{w}^{(\vect{m})}_{\vect{j}} \sum_{\vect{\gamma} \in \Gam} \frac{\polbas(\vect{i}) \overline{\polbas(\vect{j})}}{\|\polbas\|_{\mathrm{w}}^2}.\]

Evaluating the Lagrange function $L^{(\vect{m})}_{\vect{j}}$, $\vect{j} \in \Imm$,
given in \eqref{1508220009} at the nodes $(r^{(m_1)}_{i_1}, \theta^{(m_2)}_{i_2})$, $\vect{i} \in \Imm$,
and using the relation \eqref{1508201411}, we get the identity
\[L^{(\vect{m})}_{\vect{j}}(r^{(m_1)}_{i_1},\theta^{(m_2)}_{i_2}) =
\mathrm{w}^{(\vect{m})}_{\vect{j}} \sum_{\vect{\gamma} \in \Gam} \frac{\polbas(\vect{i}) \overline{\polbas(\vect{j})}}{\|\polbas\|_{\mathrm{w}}^2}.\]
Therefore, $L^{(\vect{m})}_{\vect{j}}(r^{(m_1)}_{i_1},\theta^{(m_2)}_{i_2}) = \delta_{\vect{j}}(\vect{i})$ and for $f \in \funpol(\Imm)$ we have
\[f(\vect{i}) = \sum_{\vect{j} \in \Imm} f(\vect{i}) \delta_{\vect{j}}(\vect{i}) = 
\sum_{\vect{j} \in \Imm} f(\vect{i}) L^{(\vect{m})}_{\vect{j}}(r^{(m_1)}_{i_1},\theta^{(m_2)}_{i_2})
= P^{(\vect{m})}_{f}(r^{(m_1)}_{i_1},\theta^{(m_2)}_{i_2}).\]
Thus, the function
$P^{(\vect{m})}_{f}$ solves the interpolation problem $\eqref{1508220011}$, and the mapping $f \to P^{(\vect{m})}_{f}$ is
an injective linear mapping from $\funpol(\Imm)$ into $\Pim$. Further, since $\dim \funpol(\Imm) = \dim \Pim$ this mapping is indeed
an automorphism. This implies that the interpolant $P^{(\vect{m})}_{f}$ is unique and that the system $L^{(\vect{m})}_{\vect{j}}$, $\vect{j} \in \Imm$, forms a basis of $\Pim$. Finally, if $f$ is in the subspace $\funpols(\Imm)$, we have
$f(m_1,i_2) = f_C$ for all tuples $(m_1,i_2) \in \Imm$. In this way, the interpolating function $P^{(\vect{m})}_{f}$ has the form \eqref{201513121700} and is 
contained in the subspace $\Pims$. \qed
\end{proofof}

\paragraph{\bf 9.6. Proofs of Section \ref{sec:convergence}}

\begin{proofof}{Theorem \ref{th:201807231722}} By the discussion in front of Theorem \ref{th:201807231722}, we only have to show that the continuity condition (ii) in the definition of $C(\Dd)$ is satisfied. 
The interpolant $P^{(\vect{m})}_{f} \in \Pimsquare$ is of the form
\begin{equation*} \label{eq:1807221940} P^{(\vect{m})}_{f}(r,\theta) = \sum_{\vect{\gamma} \in \Gamsquare} c_{\vect{\gamma}}(f) 
\polbascont(r,\theta) = \sum_{\gamma_2 = -m_2 + 1}^{m_2} p_{\gamma_2}(r) e^{\imath \gamma_2 \theta},
\end{equation*}
with univariate polynomials $p_{\gamma_2}(r)$, $-m_2 < \gamma_2 \leq m_2$ of degree $2 m_1$. Further, the polynomials $p_{\gamma_2}(r)$ are even if $\gamma_2$ is
even and odd otherwise. Thus, for the center $r = 0$ we obtain
\begin{equation} \label{eq:1807221941} P^{(\vect{m})}_{f}(0,\theta) = \sum_{\gamma_2 \in \{-m_2 + 1, \ldots, m_2\} \atop \gamma_2 \ \text{even}} p_{\gamma_2}(0) e^{\imath \gamma_2 \theta}.
\end{equation}
In particular, $P^{(\vect{m})}_{f}(0,\theta)$ is a $\pi$-periodic trigonometric polynomial in the $m_2$ dimensional space spanned by the functions 
$\{e^{\imath \gamma_2 \theta} \ | \ \gamma_2 \in \{-m_2+1, \ldots, m_2\}, \ \gamma_2 \; \text{even}\}$. 
Moreover, we have exactly $m_2$ different points $(0,\theta_{i_2}^{(m_2)})$, $i_2 \in \{0,2,2m_2-2\}$ in $[0,\pi)$ 
at which $P^{(\vect{m})}_{f}(0,\theta)$ is equal to the constant value $f_{\mathrm{C}} = \ff(0,0)$ given at the center of $\Dd$. These $m_2$ conditions determine the trigonometric polynomial $P^{(\vect{m})}_{f}(0,\vph)$ uniquely such that $P^{(\vect{m})}_{f}(0,\theta) = f_{\mathrm{C}}$ is constant for 
$\theta \in [-\pi,\pi]$. \qed
\end{proofof}

\begin{proofof}{Theorem \ref{thm:lebesgueconstant}}
We split the spectral index set $\Gamsquare$ into the two canonical parts
\[ \Gamsquarezero 
= \left\{\vect{\gamma} \in \Gamsquare \ | \ \gamma_1, \gamma_2 \ \text{are even}\ \right\}, \quad
\Gamsquareone
= \left\{\vect{\gamma} \in \Gamsquare \ | \ \gamma_1, \gamma_2 \ \text{are odd}\ \right\}.  \]
Then, we have $\Lambda^{(\vect{m})}_{\square} \leq \Lambda^{(\vect{m})}_{\square,0} + 
\Lambda^{(\vect{m})}_{\square,1}$, where 
\[\Lambda^{(\vect{m})}_{\square,0} = \sup_{\|\ff\|_{\infty} \leq 1} \|P^{(\vect{m})}_{f}|_{\Pi^{(\vect{m})}_{\square,0}}\|_{\infty}, \quad 
\Lambda^{(\vect{m})}_{\square,1} = \sup_{\|\ff\|_{\infty} \leq 1} \|P^{(\vect{m})}_{f}|_{\Pi^{(\vect{m})}_{\square,1}}\|_{\infty},\]
and $\Pi^{(\vect{m})}_{\square,0}$ and $\Pi^{(\vect{m})}_{\square,0}$ are the subspaces of $\Pimsquare$ 
with respect to the spectral index sets $\Gamsquarezero$ and $\Gamsquareone$, respectively. As the 
estimates for $\Lambda^{(\vect{m})}_{\square,0}$ and $\Lambda^{(\vect{m})}_{\square,1}$ are 
very similar, we will restrict all upcoming considerations to the number $\Lambda^{(\vect{m})}_{\square,0}$.

We use \eqref{20170303146} and \eqref{20170304} to reformulate $P^{(\vect{m})}_{f}|_{\Pi^{(\vect{m})}_{\square,0}}$
in terms of a double sum. We get
\begin{align*} P^{(\vect{m})}_{f}|_{\Pi^{(\vect{m})}_{\square,0}}(r,\theta) &= 
\sum_{\gamma_1 = 0}^{m_1} \sum_{\gamma_2 = - \lceil m_2/2 \rceil + 1}^{\lfloor m_2/2 \rfloor} \frac{\hat{g}(2 \vect{\gamma})}{\|\chi^{(\vect{m})}_{2 \vect{\gamma}}\|_{\mathrm{w}}^2} 
T_{2 \gamma_1}(r)  \mathrm{e}^{\imath 2 \gamma_{2} \theta}.
\end{align*}
Using the reflection symmetry $\hat{g}(\vect{\gamma}) = \hat{g}(- \gamma_1 \mod 4m_1, \gamma_2)$ of $g$ on $\Jmm$, we further get
\begin{align*} P^{(\vect{m})}_{f}|_{\Pi^{(\vect{m})}_{\square,0}}(r,\theta) &= 
\sum_{\gamma_1 = -m_1}^{m_1} \sum_{\gamma_2 = - \lceil m_2/2 \rceil + 1}^{\lfloor m_2/2 \rfloor} 
(1 - \ts \frac12 \delta_{|\gamma_1|,m_1})\hat{g}(2 \vect{\gamma}) 
\mathrm{e}^{\imath (2 \gamma_1 \arccos(r) + 2 \gamma_{2} \theta)}.
\end{align*}
For $\Lambda^{(\vect{m})}_{\square,0}$ we get in this way the bound
\begin{align*} 
\Lambda^{(\vect{m})}_{\square,0} & \leq \sup_{\|\ff\|_{\infty} \leq 1} \sup_{(r,\theta)} \left| 
\sum_{\vect{i} \in \Jmm} \sum_{\gamma_1 = -m_1}^{m_1} \!\!\! (1 - \ts \frac12 \delta_{|\gamma_1|,m_1}) \!\!\!\!\!\!\!\!\! \ds \sum_{\gamma_2 = - \lceil m_2/2 \rceil + 1}^{\lfloor m_2/2 \rfloor} \!\!\!\!\!\!  g(\vect{i}) \mathrm{e}^{-\imath 2 \gamma_{1} (i_1 \pi/m_{1} - \arccos(r)) } 
\mathrm{e}^{-\imath 2 \gamma_{2} (i_2 \pi/m_{2}-\theta) }\right| \\
&\leq  \sup_{(\theta,\vph)} \frac{1}{8 m_1 m_2} \sum_{\vect{i} \in \Jmm} \left| 
\sum_{\gamma_1 = -m_1}^{m_1} \!\!\! (1 - \ts \frac12 \delta_{|\gamma_1|,m_1}) \!\!\!\!\!\!\!\!\! \ds \sum_{\gamma_2 = - \lceil m_2/2 \rceil + 1}^{\lfloor m_2/2 \rfloor}  \mathrm{e}^{-\imath 2 \gamma_{1} (i_1 \pi/m_{1} - \arccos(r)) } 
\mathrm{e}^{-\imath 2 \gamma_{2} (i_2 \pi/m_{2}-\theta) }\right|  \\
&\leq C \int_{0}^{2 \pi} \left| \sum_{\gamma_1 = -m_1}^{m_1} \!\!\! (1 - \ts \frac12 \delta_{|\gamma_1|,m_1})
  \mathrm{e}^{- \imath 2 \gamma_{1} \rho } \right| \Dx{\rho} \int_0^{2\pi} \left| 
\sum_{\gamma_2 = - \lceil m_2/2 \rceil + 1}^{\lfloor m_2/2 \rfloor} \mathrm{e}^{-\imath 2 \gamma_{2} \theta' } \right|  \Dx{\theta'}.
\end{align*}
The last transition from the two discrete sums to the continuous integrals with a constant $C >0$ independent of $\vect{m}$ is a twofold application
of a Marcinkiewicz-Zygmund inequality, see \cite[X, Theorem 7.10]{Zygmund}. 
The two univariate integrals in the last line can be considered, up to minor modifications,
as the classical univariate Lebesgue constants in the trigonometric setting \cite[II, \S 12]{Zygmund}. 
Both can be estimated in terms of a log term such that
\begin{align*} 
\Lambda^{(\vect{m})}_{\square,0} & \leq C_{\square,0} \ln (m_1+1) \ln (m_2+1).
\end{align*}
This, together with a respective very similar estimate for $\Lambda^{(\vect{m})}_{\square,1}$ gives the statement. \qed
\end{proofof}

\section{Conclusion}
In this manuscript, we derived a novel spectral interpolation scheme for the unit disk in which samples along rhodonea curves form the set of interpolation nodes. We derived three characterizations of the rhodonea nodes. In particular, the possibility
to describe these nodes as the union of two interlacing polar grids allowed us to implement the interpolation scheme in an efficient way using fast Fourier algorithms. 

The interpolation spaces are determined by a spectral index set selecting the Chebyshev-Fourier basis. While uniqueness of the interpolation scheme can be shown for a general class of interpolation spaces, the restriction to a rectangular spectral index set turned out to be advantageous for several reasons: in this case the numerical condition number is growing slowly in the number of nodes, the interpolation scheme converges fast if the interpolated function is smooth and continuity of the interpolant can be guaranteed.
This could be verified theoretically and also in a numerical experiment. Further, the interpolation scheme was applied to obtain a Clenshaw-Curtis quadrature rule on the disk.


\end{document}